\documentclass[a4paper]{article}

\usepackage[english]{babel}

\usepackage[utf8]{inputenc}
\usepackage[utf8]{inputenc}
\usepackage{amsmath}
\usepackage{graphicx}
\usepackage{amssymb}
\usepackage{amsthm}
\usepackage{tikz-cd}
\usepackage{mathrsfs}
\usepackage[colorinlistoftodos]{todonotes}
\usepackage{enumitem}
\usepackage{yfonts}
\usepackage{ dsfont }
\usepackage{MnSymbol}
\usepackage{slashed}

\title{Metric SYZ conjecture and non-archimedean geometry}

\author{Yang Li}

\date{\today}
\newtheorem{thm}{Theorem}[section]
\newtheorem{lem}[thm]{Lemma}

\theoremstyle{definition}
\newtheorem{eg}[thm]{Example}

\newtheorem{conj}[thm]{Conjecture}
\newtheorem{cor}[thm]{Corollary}

\newtheorem{rmk}[thm]{Remark}
\newtheorem{prop}[thm]{Proposition}
\newtheorem{Def}[thm]{Definition}

\newtheorem*{Notation}{Notation}

\newtheorem*{Acknowledgement}{Acknowledgement}

\newcommand{\cf}{\emph{cf.} }

\newcommand{\R}{\mathbb{R}}
\newcommand{\C}{\mathbb{C}}

\newcommand{\N}{\mathbb{N}}
\newcommand{\Q}{\mathbb{Q}}

\newcommand{\norm}[1]{\left\lVert#1\right\rVert}

\newcommand{\mres}{\mathbin{\vrule height 1.6ex depth 0pt width
		0.13ex\vrule height 0.13ex depth 0pt width 1.3ex}}

\def\Xint#1{\mathchoice
	{\XXint\displaystyle\textstyle{#1}}%
	{\XXint\textstyle\scriptstyle{#1}}%
	{\XXint\scriptstyle\scriptscriptstyle{#1}}%
	{\XXint\scriptscriptstyle\scriptscriptstyle{#1}}%
	\!\int}
\def\XXint#1#2#3{{\setbox0=\hbox{$#1{#2#3}{\int}$ }
		\vcenter{\hbox{$#2#3$ }}\kern-.6\wd0}}

\def\dashint{\Xint-}

\begin{document}
	\maketitle

\begin{abstract}
We show that assuming a conjecture in non-archimedean geometry, then a metric formulation of the SYZ conjecture can be proved in large generality.
\end{abstract}

\section{Introduction}

The purpose of this paper is to relate a metric version of the Strominger-Yau-Zaslow (SYZ) conjecture to non-archimedean (NA) pluripotential theory. One interpretation of the SYZ conjecture \cite{SYZ} is the following:

\begin{conj}
Given a 1-parameter maximally degenerate  family of polarized $n$-dimensional  Calabi-Yau (CY) manifolds $(X_t, g_t, J_t, \omega_t, \Omega_t)$ of holonomy $SU(n)$ over the punctured disc $\mathbb{D}_t^*$, then there exist special Lagrangian $T^n$-fibrations on the generic region of $X_t$ for $0<|t|\ll 1$.
\end{conj}

This interpretation puts the CY metrics at the forefront, in contrast to alternative softer viewpoints which emphasize the algebraic, symplectic or topological aspects. 
Here the generic region means a subset of $X_t$ with almost the full percentage of the CY measure on $X_t$. This represents a compromise, as the difficulty of finding a special Lagrangian fibration on the entire $X_t$ is well appreciated since \cite{Joyce}.

The potential relevance of NA geometry to SYZ conjecture was suggested in Kontsevich and Soibelman \cite{KS}\cite{KS2}. NA pluripotential theory is taken much further by Boucksom et al \cite{Boucksomsurvey}\cite{Boucksom}\cite{Boucksom1}\cite{Boucksomsemipositive}. Impressionistically,

\begin{itemize}
\item NA geometry offers a natural language to describe the degeneration of complex manifolds into real simplicial/tropical objects.

\item It systematically encodes the combinatorics of tropical geometry.

\item It encodes much of the  birational geometry for the degeneration family.

\item
Usual notions in K\"ahler geometry such as functions, line bundles, K\"ahler metrics, complex Monge-Amp\`ere  measures, have natural (albeit exotic looking) analogues in NA geometry.

\item 
There is a natural way for CY volume measures to converge into a measure on a NA space.

\item 
An analogue of the Calabi conjecture holds in the NA context: one can solve the NA Monge-Amp\`ere equation.

\item 
Under additional hypotheses, the NA Monge-Amp\`ere measure agrees with the real Monge-Amp\`ere measure.
\end{itemize}

In short, NA spaces are the natural candidates for limits of CY metrics relevant for the SYZ conjecture. While
NA pluripotential theory has close analogy with K\"ahler geometry, 
 so far it has found no direct implication on the behaviour of degenerating CY metrics.
The goal of this paper is to show that, \emph{assuming a comparison property between NA vs. real MA equations}, then the crank of NA pluripotential theory can be turned to prove the metric SYZ conjecture in quite satisfactory generality, at least in an algebraic setup. This largely accomplishes the reduction of this metric SYZ conjecture to a problem in NA geometry.

We work over $\C$. To set the scene,

\begin{itemize}
\item 
 Let $S$ be a smooth affine algebraic curve, with a point $0\in S$.
 An \emph{algebraic degeneration family} is given by a submersive projective morphism $\pi: X\to S\setminus \{0\}$ with smooth connected $n$-dimensional fibres $X_t$ for $t\in S\setminus \{0\}$. This is in contrast with the \emph{formal} setting over the punctured formal disc $\text{Spec}(K)$ with $K=\C(\!(t)\!)$.
An algebraic degeneration induces a formal degeneration by base change.

\item  
A \emph{polarisation} is given by an ample line bundle $L$ over $X$.

\item 
We say $\pi$ is a degeneration family of \emph{Calabi-Yau} manifolds if there is a trivialising section $\Omega$ of the canonical bundle $K_X$. Over a small disc $\mathbb{D}_t$ around $0\in S$, this induces holomorphic volume forms $\Omega_t$ on $X_t$ via $\Omega= dt\wedge \Omega_t$. The \emph{normalised Calabi-Yau measure} on $X_t$ is the probability measure
\begin{equation}\label{CalabiYaumeasureeqn}
d\mu_t= \frac{   \Omega_t \wedge \overline{\Omega}_t }{   \int_{X_t} \Omega_t \wedge \overline{\Omega}_t   }.
\end{equation}
The Calabi-Yau metrics $\omega_{CY,t}$ on $X_t$ are the unique K\"ahler metrics in the class $\frac{1}{ |\log |t| |  }c_1(L)$ such that 
\begin{equation}
\frac{ \omega_{CY,t}^n }{  \int_{X_t} \omega_{CY,t}^n  }=d\mu_t.
\end{equation}

\item 
We say $\pi: X\to S\setminus \{ 0\}$ is a \emph{maximal degeneration} of Calabi-Yau manifolds if the essential skeleton has the maximal dimension $n$, and the degeneration family admits a semistable snc model over $S$. Maximal degenerations are also known as large complex structure limits (\cf section \ref{volumeasymptoteessentialskeleton} for more details).

\end{itemize}

\begin{rmk}\label{sncmodel}
Notice we have intentionally avoided completing the family at $0\in S$. That would involve the choice of a \emph{model} of $\pi: X\to S$, namely a normal flat projective $S$-scheme $\mathcal{X}$ together with an isomorphism with $X$ over the punctured curve $S\setminus \{ 0\}$. It is called an \emph{snc model} if $\mathcal{X}$ is smooth, and
the central fibre over $0\in S$  is a simple normal crossing divisor in $\mathcal{X}$. If furthermore the central fibre is reduced, it is called a \emph{semistable snc model}. Models can be analogously defined over the formal disc. The existence of snc models is a consequence of Hironaka's resolution theorem. They are highly nonunique. By the semistable reduction theorem \cite[chapter 2]{ToroidalembeddingsI}, after finite base change to another smooth algebraic curve $S'$, we can always find some semistable snc model for the degeneration family $X\times_S (S'\setminus \{0\})$, so the existence of a semistable snc model is not a substantial assumption. Everything here is quasi-projective. The choice of a model is very useful, but not intrinsic to the degenerating CY metrics.
\end{rmk}

We now briefly explain the context of the comparison property between NA vs. real MA equations, referring the details to section \ref{NAsurvey}. Given a polarized algebraic maximal degeneration family of CY manifolds, one can canonically associate a NA object called the \emph{Berkovich space} $X_K^{an}$, which can be viewed as the inverse limit of all the dual intersection complexes $\Delta_{\mathcal{X}}$ of snc models over the formal disc. There is a \emph{Lebesgue measure} $d\mu_0$ supported on the essential skeleton $Sk(X)\subset X_K^{an}$, which is the natural limit of normalised CY measures on $X_t$ in a suitable sense. A central result of NA pluripotential theory due to Boucksom-Favre-Jonsson \cite{Boucksom} is the solution of the NA Calabi conjecture. In this setting, it provides a unique (up to scale) continuous semipositive metric $\norm{\cdot}_{CY}$ on $X_K^{an}$ with the polarisation $L$, which solves the NA MA equation
\[
MA(\norm{\cdot}_{CY} )= (L^n)d\mu_0,
\]
where $(L^n)=\int_{X_t} c_1(L)^n$. The NA MA measure is defined by intersection theory, so this is not even a partial differential equation as it stands, although the theory shares many features of K\"ahler geometry. Given a model $\mathcal{L}$ of the line bundle $L$, then the semipositive metric can be represented by a potential function $\phi_0$ on $X_K^{an}$, similar to the usual relation in K\"ahler geometry between Hermitian metrics on line bundles and potential functions.

Now given any snc model $\mathcal{X}$ over the formal disc, there is a retraction map $r_{\mathcal{X}}: X_K^{an}\to \Delta_{\mathcal{X}}$ onto the dual intersection complex, and $Sk(X)$ sits inside $\Delta_{\mathcal{X}}$ as a simplicial subcomplex, with dimension $n$ by the maximal degeneration assumption. The $n$-dimensional open faces $\text{Int}(\Delta_J)$ of $Sk(X)$ have a canonical integral affine structure. In local affine coordinates, the Lebesgue measure $d\mu_0$ is a constant multiple of $dx_1\ldots dx_n$. Assuming $\mathcal{L}$ lives over $\mathcal{X}$, the semipositivity of $\norm{\cdot}_{CY}$ implies that the potential $\phi_0$ is a convex function on $\text{Int}(\Delta_J)$, similar to the fact that K\"ahler metrics are locally represented by psh functions. Consequently, the restriction of $\phi_0$ to $\text{Int}(\Delta_J)\subset Sk(X)$ has a well defined real MA measure $MA_\R(\phi_0)$.

The comparison property hypothesis (\cf section \ref{NACalabi}) requires that there is some semistable snc model $\mathcal{X}$ over the formal disc as above, such that over all the $n$-dimensional open faces $\text{Int}(\Delta_J)$ of $Sk(X)\subset \Delta_{\mathcal{X}}$, the function $\phi_0= \phi_0\circ r_{\mathcal{X} }$ factors through the retraction map. This means the information of the potential $\phi_0$ on the abstract looking space $X_K^{an}$ is largely contained in a convex function on these open faces. Morever, a recent result of Vilsmeier \cite{Vilsmeier} implies that under this comparison hypothesis, then $\phi_0$ satisfies a real MA equation on these open faces:
\[
MA_\R(\phi_0)= \frac{ (L^n) }{n!}d\mu_0.
\]
In particular, the NA MA equation is closely related to a PDE after all. Such a relation between NA geometry and the real MA equation is in the spirit of the Kontsevich-Soibelman conjecture \cite{KS} about Gromov-Hausdorff limits of maximally degenerate CY metrics.

The main theorem of this paper is

\begin{thm}
Let $X\to S\setminus \{ 0\}$ be an algebraic maximal degeneration family of Calabi-Yau manifolds, with the polarization ample line bundle $L\to X$. Assume the NA MA-real MA comparison property holds for $X$. 
Given any $0<\delta\ll 1$, for sufficiently small $t$ depending on $\delta$, there exists a special Lagrangian $T^n$-fibration with respect to the Calabi-Yau structure $(\omega_{CY,t}, \Omega_t)$ on an open subset of $X_t$ whose normalized Calabi-Yau measure is at least $1-\delta$.
\end{thm}

The proof strategy is the following. It has already been understood in the author's previous paper \cite{LiFermat} that for the existence of the special Lagrangian fibration in the generic region, it is enough to prove that the K\"ahler potential for the CY metric on $X_t$ is $C^0$-close to a solution of the real Monge-Amp\`ere equation at least in the generic region. The natural candidate of this real MA solution comes from the Boucksom-Favre-Jonsson solution of the NA MA equation. 
To bridge the gap between the NA space and the complex manifold $X_t$, we go through a Fubini-Study $ C^0$-approximation of the potential, to get a K\"ahler metric whose potential  is $C^0$ close to the BFJ solution in a suitable sense. Here it is crucial to preserve positivity. We then regularize it to make its volume form close to being CY in the generic region. We compare the regularized potential to the CY potential in $C^0$ by adapting an $L^1$-stability argument of Kolodziej; making this work on the  highly degenerate complex manifolds $X_t$ requires a uniform Skoda estimate, proved in a companion paper \cite{LiuniformSkoda}.

In this strategy the appeal to NA geometry is for the following reasons.

\begin{itemize}
\item NA geometry is a natural language to discuss intrinsic properties of the degeneration family independent of the choice of models.

\item  The NA MA solution is a natural candidate solution to the real MA equation on the essential skeleton $Sk(X)$.


\item  The NA framework effectively encodes the notion of positivity (psh properties). This gives one answer to the question: what is the appropriate analogous notion of K\"ahler potentials in the maximal degeneration limit? 

\item
The solution to the NA MA equation is known to be unique. This allows one to expect that the limit of the CY metrics on $X_t$ is unique at least in the generic region, without the need to pass to subsequences.

\end{itemize}

The organization of the paper is as follows. We collect some essential analytic ingredients in section \ref{analyticalbackgrounds}; in particular we prove a uniform one-sided version of the $L^1$-stability estimate by adapting an argument of Kolodziej. Section \ref{NAsurvey} is an introduction of NA geometry for differential geometers, and the emphasis is on the relation with K\"ahler geometry. Section \ref{PotentialestimateSYZfibration} works on the metric SYZ conjecture proper, and the key is to prove the $C^0_{loc}$-convergence of the local potentials of the CY metric towards the NA solution, assuming the comparison property. We end this section with some discussions about the closely related Kontsevich-Soibelman conjecture. Section \ref{Furtherdirections} presents problems and speculations beyond the SYZ setting. It contains a heuristic general formula for NA MA measure, and suggests how the NA MA equation is related to non-maximal degenerations by proposing a generalised Calabi ansatz.


\begin{Notation}
	Our convention is $d=\partial+ \bar{\partial}$,  $d^c= \frac{ \sqrt{-1} }{2\pi} (-\partial+ \bar{\partial})$, so $dd^c= \frac{\sqrt{-1}}{\pi}\partial \bar{\partial}$. The relation between K\"ahler potentials and K\"ahler metrics is $\omega_\phi= \omega+dd^c\phi$. Alternatively, we think of a K\"ahler metric in terms of local absolute potentials, meaning $\omega=dd^c\varphi$ for locally defined psh functions $\varphi$. Given a Hermitian metric $h$ on a line bundle $L$, its curvature form is $-dd^c\log h^{1/2}$ in the class $c_1(L)$.
\end{Notation}

\begin{Acknowledgement}
The author is a 2020 Clay Research Fellow, currently based at the Institute for Advanced Study.
He thanks S. Boucksom and C. Vilsmeier for answering questions on NA geometry, and Song Sun, Simon Donaldson, and Valentino Tosatti for discussions.
\end{Acknowledgement}

\section{Analytic backgrounds}\label{analyticalbackgrounds}

\subsection{Uniform Skoda inequality}\label{Skodainequalitybackgroundsection}

Given a K\"ahler manifold $(Y,\omega)$, 
an upper semicontinuous $L^1_{loc}$ function $\phi\in PSH(Y,\omega)$, if $\omega_\phi=\omega+ dd^c \phi\geq 0$. A Skoda type inequality captures the apriori regularity of such functions. The following uniform version is the main result in the author's companion paper \cite{LiuniformSkoda}.

\begin{thm}(Uniform Skoda estimate)\label{UniformSkodathm}
Given a polarised algebraic  degeneration family of Calabi-Yau manifolds $\pi: X\to S\setminus \{0\}$ as in the Introduction. Let $\omega_{FS}$ be a fixed Fubini-Study metric on $(X,c_1(L))$ induced by a projective embedding via the sections of a high power of $L$, and use $\omega_{FS,t}= \frac{1}{|\log |t||} \omega_{FS}|_{X_t}$ to define a family of background metrics on $X_t$ in the class  $\frac{1}{|\log |t||} c_1(L)$. Then there are uniform positive constants $\alpha, A$ independent of $t$ for $0<|t|\ll 1$, such that for the normalised Calabi-Yau measures $d\mu_t$,
	\[
	\int_{X_t} e^{-\alpha u}  d\mu_t \leq A, \quad \forall u\in PSH(X_t, \omega_{FS,t}) \text{ with } \sup_{X_t} u=0.
	\]

\end{thm}

\subsection{Kolodziej's estimate on pluripotentials}\label{Toolsfrompsh}

Given an $n$-dimensional K\"ahler manifold $(Y,\omega)$, for $\phi\in PSH(Y,\omega)\cap L^\infty$, pluripotential theory allows one to make sense of the Monge-Amp\`ere (MA) measure $\omega_\phi^n$, generalising the notion of volume forms. A basic problem is to estimate $\phi$  from a priori bounds on $\omega_\phi^n$. A prototypical result is (\cf \cite[section 2.2]{LiFermat} for an exposition based on \cite{EGZ}\cite{EGZ2}):

\begin{thm}\label{pluripotentialthm1}
	Let $(Y, \omega)$ be a compact K\"ahler manifold, and $\phi\in PSH(Y,\omega)\cap C^0$, such that $\omega_\phi^n$ is an absolutely continuous measure. Assume there are positive constants $\alpha, A$, such that the Skoda type estimate holds with respect to $\omega_\phi^n$:
	\begin{equation}\label{Skodaassumption}
	\int_Y e^{-\alpha u}  \frac{\omega_\phi^n }{ \text{Vol}(Y) } \leq A, \quad \forall u\in PSH(Y,\omega) \text{ with } \sup_Y u=0.
	\end{equation}
\begin{itemize}

\item   For fixed $n,\alpha,A$, there is number $B(n,\alpha,A)$, such that if $\frac{ \int_{ \phi\leq -t_0} \omega_\phi^n}{  \text{Vol}(Y)  }< (2B)^{-2n}$ for some $t_0$, then $\min \phi\geq -t_0- 4B (\frac{ \int_{ \phi\leq -t_0} \omega_\phi^n}{  \text{Vol}(Y)} )^{1/2n}$.  

\item   If $\sup_Y \phi=0$, then $\norm{\phi}_{C^0} \leq C(n,\alpha,A)$.
\end{itemize}

\end{thm}

The strength of this result is that it still applies when the complex/K\"ahler structures are highly degenerate, as it distills the dependence on $(Y,\omega)$ to only 3 constants $n,\alpha,A$. A minor variant gives a criterion for two K\"ahler potentials to be close to each other.

\begin{cor}\label{StabilityestimateKolodziej}
	(Stability estimate) Let $(Y,\omega)$ be a compact K\"ahler manifold, and $\phi, \psi\in PSH(Y,\omega)\cap C^0$, such that $\omega_\psi^n$ is absolutely continuous. Assume $\norm{\phi}_{C^0} \leq A'$ and the Skoda type estimate (\ref{Skodaassumption}). Then there is a number $B(n,A, A', \alpha)$, such that if $\frac{ \int_{ \psi-\phi\leq -t_0} \omega_\psi^n}{  \text{Vol}(Y)  }< (2B)^{-2n}$ for some $t_0$, then \[
	\min (\psi-\phi)\geq -t_0- 4B \left(\frac{ \int_{ \psi-\phi\leq -t_0} \omega_\psi^n}{  \text{Vol}(Y)} \right)^{1/2n}.
	\].   
\end{cor}

\subsection{$L^1$-stability estimate}

In complex pluripotential theory, an $L^1$-stability estimate is an assertion about the $C^0$-closeness of two K\"ahler potentials given that their volume densities are close in $L^1$. We now adapt an argument of Kolodziej \cite{KolodziejL1} to prove a uniform version which allows the complex structure to be highly degenerate and the volume to collapse. Our formulation also brings out the asymmetrical role of the two K\"ahler potentials; in fact one of them is set to zero. We do not pursue optimality.

\begin{lem}\label{Comparisonprinciple}
(Comparison principle)\cite[Thm. 2.1]{KolodziejL1}
If $u$ and $v$ are $\omega$-psh on $Y$, then on $\Omega=\{ u<v  \}$, we have
\[
\int_{\Omega} \omega_u^n\geq \int_{\Omega} \omega_v^n.
\]
\end{lem}

\begin{lem}\label{Concavitylem}
(Concavity of $det^{1/n}$) On an open domain, suppose $u, v$ are continuous $\omega$-psh functions, with 
\[
\omega_u^n = f\omega^n ,  \quad \omega_v^n= g\omega^n
\]
for $f, g\in L^\infty$. 
Then for $0<s<1$, we have $\omega_{su+(1-s)v}^n \geq (sf^{1/n}+ (1-s)g^{1/n})^n\omega^n$.
\end{lem}

\begin{proof}
In the smooth case this is a pointwise inequality expressing the concavity of $A\mapsto \det^{1/n} A$ on the set of Hermitian matrices. In general one shows this by an approximation argument \cite[Lemma 1.2]{KolodziejL1}.
\end{proof}

\begin{thm}\label{UniformL1stabilitythm}
(Uniform $L^1$-stability) Let $(Y,\omega)$ be a compact K\"ahler manifold, and $\phi\in PSH(Y,\omega)\cap C^0$, satisfying the complex MA equations
\[
\frac{\omega^n }{ \text{Vol}(Y) }=d\mu,  \quad \frac{\omega_\phi^n }{ \text{Vol}(Y) }=d\nu
\]
for probability measures $d\mu$ and $d\nu$.
Assume
\begin{itemize}
\item There is a Skoda estimate
\[
\int_Y e^{-\alpha u}  d\mu \leq A, \quad \forall u\in PSH(Y,\omega) \text{ with } \sup_Y u=0.
\]

\item The complement of $E_0=\{ \phi>0  \}$ has a mass lower bound \[ \int_{  E_0^c} d\mu \geq \lambda>0. \]

\item ($L^1$-stability assumption) The total variation $\int_Y |d\mu-d\nu| \leq s^{2n+3}<1$. 

\item $\phi$ is smooth away from a (possibly empty) closed subset $S$ with $d\mu$-measure zero. Globally $\norm{\phi}_{C^0}\leq A'$.
\end{itemize}
Then for $0<s<s_0(\lambda, n, \alpha, A, A')\ll 1$, there is a uniform estimate 
\[
\sup_Y \phi \leq C(\lambda, n, \alpha, A, A') s.
\]
\end{thm}

\begin{proof}
We can reduce to the case with $d\mu(E_0)\geq \frac{1}{2}$ by shifting $\phi$ by a constant.
We construct an auxiliary continuous $\omega$-psh function $\rho$ by solving the complex MA equation with $L^\infty$-density \cite{EGZ}
\[
\frac{ \omega_\rho^n}{ \text{Vol}(Y) } = \frac{1}{ d\mu(E_0) } d\mu \mres  E_0, \quad \sup_{Y} \rho=0.
\]
where $d\mu \mres  E_0(F)= d\mu(E\cap F)$ is the restricted measure. By Theorem \ref{pluripotentialthm1} we have $\norm{\rho}_{C^0} \leq C(n,\alpha, A)$, since the RHS measure satisfies a Skoda estimate. We choose $a=\max(C(n,\alpha, A), A'  )$, so that
\[
-a\leq \rho\leq 0, \quad \phi\geq -a,
\]
implying the set inclusion
\[
E'=\{  \phi>2as+ s \max \phi   \} \subset E=\{ (1-s)\phi+s\rho-as>0  \} \subset E_0.
\]
Our next goal is to show $d\mu(E')$ is small.

Let $G=\{ 1-s^2 \geq \frac{ d\nu}{d\mu}  \}\subset Y\setminus S$. On the open set $E_0\setminus (S\cup G)$, by the concavity Lemma \ref{Concavitylem},
\[
\frac{ \omega_{s\rho+(1-s)\phi}^n }{  \text{Vol}(Y)}\geq \left( sd\mu(E_0) ^{-1/n} + (1-s)(1-s^2)^{1/n}  \right)^n d\mu.
\] 
Now $d\mu(E_0)^{-1/n}\geq (1-\lambda)^{-1/n}$ by assumption. Choose $0<q< n\{ (1-\lambda)^{-1/n}-1\}$, so for $0<s\ll 1$ depending on $\lambda, n$, by Taylor expansion in $s$,
\[
\left( sd\mu(E_0) ^{-1/n} + (1-s)(1-s^2)^{1/n}  \right)^n \geq 1+qs.
\] 	
Combining this with the comparison principle Lemma \ref{Comparisonprinciple}, and the assumption $d\mu(S)=0$,
\[
(1+qs) \int_{E\setminus G} d\mu \leq \int_{E}\frac{ \omega_{s\rho+(1-s)\phi}^n }{  \text{Vol}(Y)} \leq \int_E \frac{ \omega^n }{  \text{Vol}(Y)}= \int_E d\mu.
\]	
On the other hand, by the definition of $G$ and the $L^1$-stability assumption, 
\[
d\mu(G) \leq s^{-2} \int_G (d\mu-d\nu) \leq s^{-2}\int_Y |d\mu-d\nu|\leq s^{2n+1},
\]
hence
\[
(1+qs) \int_{E\setminus G} d\mu \leq \int_{E\setminus G} d\mu+ s^{2n+1}.
\]
We conclude $\int_{E\setminus G} d\mu \leq q^{-1} s^{2n}$, so
\[
\int_{E'} d\mu \leq \int_E d\mu \leq q^{-1}s^{2n}+ s^{2n+1} \leq (q^{-1}+1) s^{2n}. 
\]

We now apply the stability estimate Cor. \ref{StabilityestimateKolodziej} to compare the potentials $\phi$ and $0$, to see for $s$ sufficiently small depending on $n,A,A',\alpha, \lambda$,
\[
\min_Y(-\phi)\geq -2as-s \max_Y \phi- 4B  ( \int_{E'} d\mu)^{1/2n}\geq -2as-s \max_Y \phi- 4B(1+q^{-1})^{1/2n}s,
\]
whence $\max_Y \phi \leq C(n, A, A', \alpha, \lambda)s$ as required.
\end{proof}

\begin{rmk}
It is not clear to the author why in Kolodziej's original argument \cite[Lemma 1.2]{KolodziejL1} applies  in the proof of \cite[Thm. 4.1]{KolodziejL1}, as $E_0\setminus G$ is not an open domain if one considers general $L^p$-densities.
\end{rmk}

\subsection{Savin's small perturbation theorem}

Savin \cite{Savin} proved that for a large class of second order elliptic equations satisfying certain structural conditions, any viscosity solution $C^0$-close to a given smooth solution has interior $C^{2,\gamma}$-bound. In particular this applies to complex MA equation. Combined with the Schauder estimate,

\begin{thm}\label{Savin}
Fix $k\geq 2$ and  $0<\gamma<1$.
On the unit ball, let $v$ be a given smooth solution to the complex Monge-Amp\`ere equation $(dd^c v)^n=1$. Then there are constants $0<\kappa\ll 1$ and $C$ depending on $n, k,\gamma, \norm{v}_{C^{k,\gamma}}$, such that if 
\[
(dd^c (u+v))^n=1+f, \quad \norm{f}_{C^{k-2,\gamma}}<\kappa,
\]
and $\norm{u}_{C^0}<\kappa$, then $\norm{u}_{C^{k,\gamma}(B_{1/2})}\leq C\kappa$.
\end{thm}

\subsection{Regularity theory for real Monge-Amp\`ere}\label{RegularitytheoryforrealMA}

There is an extensive literature on the local regularity theory for the real Monge-Amp\`ere equation, largely due to the Caffarelli school. The author thanks C. Mooney for bringing some of these results to his attention. All results surveyed here can be found in \cite{Mooney}.

Any convex function on an open set $v: \Omega\subset \R^n\to \R$ has an associated Borel measure called the Monge-Amp\`ere measure, defined by
\[
MA_\R(v)(E)= |\partial v (E)|,
\]
where $|\partial v(E)|$ denotes the Lebesgue measure of the image of the subgradient map on $E\subset \Omega$. Given a Borel measure $\mu$, a solution to $MA(v)=\mu$ is called an \emph{Aleksandrov solution} to 
$
\det(D^2 v)= \mu;
$ 
if $v\in C^2$, this is the classical real Monge-Amp\`ere equation. 
We shall assume a two-sided density bound
\[
\det(D^2 v)=f \text{ in }B_1, \quad 0<\Lambda_1\leq f\leq \Lambda_2.
\]
Let $B_1\setminus \Sigma$ be the set of strictly convex points of $v$, namely there is a supporting hyperplane touching the graph of $v$ only at one point. Then Caffarelli \cite{Caff1}\cite{Caff2}\cite{Caff4} shows
\begin{itemize}
\item
If $f\in C^\gamma(B_1)$, then $v\in C^{2,\gamma}_{loc}(B_1\setminus \Sigma)$. Then by Schauder theory, if $f$ is smooth, then $v$ is smooth in $B_1\setminus \Sigma$.
\item
If $L$ is a supporting affine linear function to $v$, such that the convex set $\{v=L\}$ is not a point. Then $\{v=L\}$ has no extremal point in the interior of $B_1$. 
\item
The above affine linear set $\{v=L\}$ has dimension $k< n/2$. 
\end{itemize}

Mooney \cite{Mooney} shows further that
\begin{itemize}
\item The singular set $\Sigma$ has $(n-1)$-Hausdorff measure zero. Consequently $B_1\setminus \Sigma$ is path connected (because a generic path joining two given points does not intersect a subset of zero $(n-1)$-Hausdorff measure).

\item  The solution $v\in W^{2,1}_{loc}(B_1)$ even if $\Sigma$ is nonempty.
\end{itemize}

\begin{rmk}\label{Mooneyremark}
A classical counterexample of Pogorelov shows that for $n=3$, the singular set $\Sigma$ can contain a line segment. This is generalised by Caffarelli \cite{Caff4}, who for any $k<n/2$ constructs examples where $f$ is smooth but $\Sigma$ contains a $k$-plane. A surprising example of Mooney \cite{Mooney} shows that the Hausdorff dimension of $\Sigma$ can be larger than $n-1-\epsilon$ for any small $\epsilon$. This means the local regularity theory surveyed above is essentially optimal.
\end{rmk}

\section{Nonarchimedean geometry}\label{NAsurvey}

This section is a differential geometer's quick tour into the wonderland of NA geometry. Our goal is to motivate the basic concepts from natural problems in K\"ahler geometry, and thereby build up a dictionary between K\"ahler geometry and NA geometry, instead of giving a systematic survey. The author's viewpoint is heavily influenced by Boucksom et al. \cite{Boucksomsurvey}\cite{Boucksom}\cite{Boucksom1}\cite{Boucksomsemipositive}.
He thanks S. Boucksom and C. Vilsmeier for explanations.

The heuristic mental picture is the following. Complex algebraic geometry has two `orthogonal' transcendental limits. When we look at very high powers of an ample line bundle $L$, the rescaled Fubini-Study metrics can be used to approximate any K\"ahler metric in $c_1(L)$, so \emph{K\"ahler geometry/complex pluripotential theory is the asymptotic limit of complex algebraic geometry}. When we fix a polarisation and degenerate the complex structure, much of the information is encoded by piecewise linear objects appearing from logarithm maps, so \emph{tropical geometry is the tropical limit of complex algebraic geometry}. As we climb high up the tower of snc models, then piecewise linear convex functions can be used to approximate more general convex functions, so \emph{the asymptotic limit of tropical geometry is NA geometry/NA pluripotential theory}. As the reader can guess by completing this Cartesian square, \emph{NA pluripotential theory is the tropical limit of K\"ahler geometry}, because psh functions on a large annulus inside $(\C^*)^n$ resemble convex functions on large domains in $\R^n$. Building more links between NA pluripotential theory and K\"ahler geometry is in fact the main goal of this paper.

\subsection{Volume asymptote and essential skeleton}\label{volumeasymptoteessentialskeleton}

Consider an algebraic Calabi-Yau degeneration family $X\to S\setminus \{0 \}$ as in the Introduction. We are given the holomorphic volume forms $\Omega_t$ on the fibres $X_t$, and let us follow \cite{Boucksom1} to consider the question of calculating the asymptote of $\int_{X_t} \Omega_t \wedge \overline{\Omega}_t$ as $t\to 0$. Since we only care about small $t$, we are free to shrink $S$. For instance, we may assume $dt$ is nowhere vanishing on $S$.

A very useful tool is to fill in the central fibre by choosing an \emph{snc model} (\cf Remark \ref{sncmodel}) $\mathcal{X}$ over $S$. The central fibre $\mathcal{X}_0$ is an snc divisor with components $E_i$ for $i\in I$, and we write $\mathcal{X}_0= \sum_{i\in I} b_i E_i$. In the special case of \emph{semistable snc models} $b_i=1$ for $i\in I$; this can always be achieved after finite base change. The canonical divisor $K_{\mathcal{X}}$ is supported on $\mathcal{X}_0$ as $K_X$ has a trivialising section $\Omega$. We may write $K_{\mathcal{X} }=\sum_i (a_i+b_i-1 )E_i $, so that the \emph{relative log canonical divisor}
\[
K^{log}_{\mathcal{X}/S  }:= K_{\mathcal{X}}- K_S + \mathcal{X}_{0,red}- \mathcal{X}_0=    \sum a_i E_i.
\]
Shifting all $a_i$ by a constant $\kappa$ is equivalent to multiplying $\Omega$ by $t^\kappa$, which gives an elementary factor $|t|^{2\kappa }$ to $\int_{X_t} \Omega_t \wedge \overline{\Omega}_t$. Thus we shall always assume $\min a_i=0$.

It is useful to introduce a \emph{quantitative stratification} on $X_t$ according to the intersection pattern of $E_i$. Let $E_J=\cap_{i\in J} E_i$ for $J\subset I$, which is irreducible if nonempty. Using the distance function of a fixed smooth background K\"ahler metric on $\mathcal{X}$, we can write 
\[
E_J^0=\{  q\in X_t| d(q, E_J) \ll 1            \} \setminus \{   q\in X_t| d(q, E_{J'}) \ll 1   , \quad \text{some } J'\supsetneq J           \}.
\]
Around $\emptyset \neq E_J\subset \mathcal{X}$, we denote $p=|J|-1$,
and introduce local coordinates $z_0, \ldots z_n$ on $\mathcal{X}$, such that $z_0, z_1, \ldots, z_p$ are the defining equations of $E_i$ for $i\in J$. The conditions on the divisors mean that away from deeper strata we may arrange $t= \prod_0^p z_i^{b_i}$, and
\[
\Omega= u_J  \prod_0^p z_i^{a_i+b_i} d\log z_i \wedge \prod_{p+1}^n dz_j
\]
for some local nowhere vanishing holomorphic function $u_J$. By definition $\Omega=dt \wedge \Omega_t$ along $X_t$, so on $E_J^0$
\[
\Omega_t = b_0^{-1} u_J z_0^{a_0}\ldots z_p^{a_p} \prod_1^p d\log z_i \wedge \prod_{p+1}^n dz_j,
\]
\[
\sqrt{-1}^{n^2} \Omega_t\wedge \overline{\Omega}_t = |b_0|^{-2} |u_J|^2 |z_0|^{2a_0 }\ldots |z_p|^{2a_p} \prod_1^p \sqrt{-1} d\log z_i \wedge d\log \bar{z}_i \wedge \prod_{p+1}^n \sqrt{-1} dz_j \wedge d\bar{z}_j.
\]
Notice also that the local equation $t= \prod_0^p z_i^{b_i}$ has $b_J= \gcd_{i\in J} b_i$ sheets of solutions. Using the polar coordinates by $z_i= e^{ x_i\log |t|  + \sqrt{-1}\theta_i }$ for $i\in J$, ones sees that the magnitude of $\int_{E_J^0} \sqrt{-1}^{n^2} \Omega_t\wedge \overline{\Omega}_t$ is  $O(|\log |t| |^l  )$ for $l= |\{ j\in J: a_j=0  \} |-1$.

The local logarithmic variables $x_i=\frac{ \log |z_i|}{ \log |t| }$ lie on the simplex
\[
\Delta_J= \{  \sum_0^p b_i x_i =1, \quad 0\leq x_i \leq 1 \}.
\]
These depend on the choice of $z_i$, but since the local defining equation of divisors differ by a nowhere vanishing holomorphic function, the ambiguity of $x_i$ is only $O( \frac{1}{ |\log |t|| } )$ for $0<|t|\ll 1$. Taking a more global viewpoint, the combinatorial pattern of how these simplices fit together exactly reflects the intersection pattern of the divisors $E_i$. Formally, this information is encoded in the \emph{dual intersection complex} $\Delta_{\mathcal{X}}$ for the snc model $\mathcal{X}$: this is the polyhedral complex whose vertices $v_i$ correspond to $E_i$, and we assign a simplex $\Delta_J$ with vertices $v_i$ for $i\in J$ if and only if $E_J\neq 0$. The coodinates $x_j$ then define a \emph{piecewise integral affine structure} on $\Delta_{\mathcal{X}}$. Up to the above $O(\frac{1}{ |\log |t||  }  )$ ambiguity, we now have a \emph{logarithm map} $\text{Log}_{\mathcal{X}}: X_t\to \Delta_{\mathcal{X}}$, locally described by $x_i=\frac{ \log |z_i|}{ \log |t| }$. Consequently, the `hybrid' space $X  \sqcup \Delta_{\mathcal{X}}$ is equipped with a natural topology, so that a sequence of points $z_k\in X_t$ converges to $x\in \Delta_{\mathcal{X}}$ iff $t\to 0$ and $\text{Log}_{\mathcal{X} }(z_k)\to x$. 
   The name `hybrid' refers to the mixture of algebraic varieties with simplicial objects. The measure calculation above explains why $\Delta_{\mathcal{X}}$ (or rather the essential skeleton inside, see below) is a better candidate notion as a limit of $X_t$ than the algebraic limit $\mathcal{X}_0$. Intuitively, the Calabi-Yau measure in the limit becomes mutually orthogonal with the measure of any fixed Fubini-Study metric, and the region carrying most of CY measure looks `small' from the algebraic perspective.

The measure also singles out a distinguished subcomplex $Sk(\mathcal{X})$, called the \emph{essential skeleton}, consisting of the simplices in $\Delta_{\mathcal{X} }$ whose vertices correspond to $E_i$ with $a_i=0$. This is where the limit of the normalised CY measure is supported. The dimension of $Sk(\mathcal{X})$ is a measurement of how transcendental the degeneration $X$ is; it is reflected by the growth order of $\int_{X_t} \Omega_t\wedge \overline{\Omega}_t$. In the case of a \emph{maximal degeneration}, $\dim_\R Sk(\mathcal{X})=n$. Let us analyze the CY measure more explicitly for maximal degenerations, in a semistable snc model. For $E_J$ corresponding to an $n$-dimensional simplex in $Sk(\mathcal{X})$, on $E_J^0$
\begin{equation}
\sqrt{-1}^{n^2} \Omega_t\wedge \overline{\Omega}_t =  |u_J|^2  \prod_1^n \sqrt{-1} d\log z_i \wedge d\log \bar{z}_i.
\end{equation}
Here $u_J$ limits to its value $u_J(E_J)$ at the point stratum $E_J$, which is called the Poincar\'e residue of $\Omega$, and is easily seen to be independent of the choice of coordinates $z_i$. It is a consequence of the residue theorem on Riemann surfaces that $|u_J(E_J)|^2$ is independent of such $J$ \cite[Thm. 7.1]{Boucksom1}. Thus the pushforward to $\Delta_{\mathcal{X}}$ of the normalised CY measure (\ref{CalabiYaumeasureeqn})  converges smoothly in the interior of $\Delta_J$ to a constant multiple of the \emph{Lebesgue measure}:
\begin{equation}\label{measureconvergence}
\text{Log}_{ \mathcal{X} *  }d\mu_t= \text{Log}_{ \mathcal{X}*  }\frac{ \Omega_t\wedge \overline{\Omega}_t }{  \int_{X_t}  \Omega_t\wedge \overline{\Omega}_t    } \xrightarrow{t\to 0}    d\mu_0:=\text{Const}\cdot dx_1\ldots dx_n.
\end{equation}
Notice $dx_1\ldots dx_n$ is canonically defined due to the presence of an integral affine structure on $\Delta_J$. Viewed as a measure on $\Delta_{\mathcal{X}}$, the limit $d\mu_0$ has null measure on the complement of the $n$-dimensional faces of $Sk(\mathcal{X})$, as the integral of $d\mu_t$ in the corresponding region is $O(\frac{1}{ |\log |t|| } )$. The constant in (\ref{measureconvergence}) is independent of $J$ and its sole purpose is to make $d\mu_0$ a probability measure.

\subsection{Berkovich space, hybrid topology}\label{Berkovichspace}

One problem with the dual intersection complex is that it involves a \emph{choice} of an snc model. The snc models are highly nonunique: we can keep blowing up to get a directed set of snc models. One would like to extract intrinsic information about the degeneration family. There are two general strategies. First, one can analyze the relation between different models and seek an optimal choice using the minimal model program \cite{NicaiseXu}\cite{NicaiseXuYu}; this usually leaves the snc world, and even the optimal choices may still be nonunique. Alternatively we can consider all (snc) models simultaneously, by the language of NA geometry. Good references can be found in \cite[A]{KS}  \cite[Appendix]{Boucksom1}\cite[chapter 2,3]{Boucksomsemipositive}.

An insight of Berkovich is that by thinking of points as multiplicative seminorms, one obtains a kind of geometry analogous to complex manifolds.
Let  $K\simeq\C(\!(t)\!)$ be equipped with its standard absolute value $|\cdot |_0=e^{-ord_t}$ where $ord_t$ is the valuation defined by the vanishing order. Its ultrametric property 
\[
|f+g|_0\leq \max\{ |f|_0, |g|_0  \}
\]
gives the name `non-archimedean' to the subject. Let $X_K$ be a smooth, geometrically connected, projective scheme over $\text{Spec}(K)$; the main examples come from base changing an algebraic degeneration family $X$ over a punctured curve. Choose a finite cover of $X_K$ by affine open sets of the form $U=\text{Spec}(A)$, where $A$ is a finitely generated $K$-algebra. The space $U^{an}$ is defined as the set of all multiplicative seminorms $|\cdot|_x: A\to \R_{\geq 0}$ extending the absolute value of $K$, endowed with the weakest topology so that the function $x\mapsto |f|_x $ is continuous  for any $f\in A$. The \emph{Berkovich space} $X_K^{an}$ is then obtained by gluing together $U^{an}$; the notation stands for `analytification'. As a topological space $X_K^{an}$ is compact and Hausdorff. In the CY case, the point-set description of $X_K^{an}$ is meant to encode information about the base of the SYZ fibration; there is also a natural structure sheaf which encodes information about the complex structure
\cite{KS}.

Let $R\simeq \C[\![t ]\!]$. The concept of models over $\text{Spec}(R)$ is entirely analogous to the case over algebraic curves. The dual intersection complexes $\Delta_{\mathcal{X}}$ for snc models over $\text{Spec}(R) $ can be compared with $X_K^{an}$ through two natural maps:

\begin{itemize}
\item There is a continuous \emph{embedding map} $emb: \Delta_{\mathcal{X} }\to X_K^{an}$. Writing $\mathcal{X}_0=\sum b_i E_i$, each divisor $E_i$ defines $val_{E_i}= \frac{ord_{E_i} }{ b_i  }$ through the vanishing order $ord_{E_i}$, so that $e^{- val_{E_i} }$ is a point in $X_K^{an}$, called a \emph{divisorial point}. More generally, given a point $x=(x_0,\ldots x_p)$ in the interior of a face $\Delta_J\subset \Delta_\mathcal{X}$ corresponding to $E_J=\cap_0^p E_i$, we can associate a \emph{monomial valuation}: expanding any local function $f$ around $E_J$ in Taylor series,
\[
f=\sum_{ \alpha\in \N^{p+1} } f_\alpha z_0^{\alpha_0}\ldots z_p^{\alpha_p}, \quad f_\alpha\in K(E_J)
\]
then the monomial valuation is
\[
val_x(f) = \min \{  \sum_0^p \alpha_i x_i | f_\alpha\neq 0       \}.
\] 
Thus $x$ gives rise to a point $e^{-val_x}\in X_K^{an}$. We shall regard $\Delta_{\mathcal{X}}$ as a subset of $X_K^{an}$.

\item
There is a continuous \emph{retraction map}  $r_{ \mathcal{X} }: X_K^{an}\to \Delta_{\mathcal{X} } $, which restricts to the identity on $\Delta_{  \mathcal{X}}\subset X^{an}$. 
 Any point $e^{-v}\in X_K^{an}$ admits
a center on $\mathcal{X}$. 
This is the unique scheme theoretic point
$\xi \in X_0$ such that 
$|f|_x \leq 1$ for $f\in \mathcal{O}_{\mathcal{X},\xi}$
and $|f|_x < 1$ for
 $f\in m_{\mathcal{X},\xi}$.
  Let $J \subset I$ 
 be the maximal subset such that
   $\xi\in E_J$. Then
$r_{ \mathcal{X} }(x) \in \Delta_{\mathcal{X} }$ 
corresponds to the monomial valuation with the same value for
 $-\log |z_j|_x, j \in J$.

\end{itemize}

\begin{rmk}
The retraction map depends on the choice of the \emph{model}. There are examples where two models $\mathcal{X}$ and $\mathcal{X}'$ define the same $\Delta_{\mathcal{X}}$ as a subset of $X_K^{an}$, but the retraction maps are different \cite[Appendix]{GublerJill}.
\end{rmk}

With these comparison maps, the Berkovich space $X_K^{an}$ is homeomorphic to the inverse limit of the dual intersection complexes of the snc models:
\[
X_K^{an}\simeq  \varprojlim_{\text{snc models} } \Delta_{ \mathcal{X} }
\]
Conceptually, \emph{an snc model gives a finite approximation of the Berkovich space}.

While dual intersection complexes depend strongly on the model, in the CY case the embedding image of the essential skeleton $Sk(\mathcal{X})\subset X_K^{an}$ as a set is independent of the model, so can be written as $Sk(X)$. This can be expected as $Sk(\mathcal{X})$ should support the limiting normalised CY measure, a property independent of the model choice. However, as we blow up snc models, the essential skeleton as a simplical complex can be subdivided.


We now indicate how NA geometry is unified with complex geometry. Consider an algebraic degeneration $X$ over a punctured curve. Let $|\cdot|$ denote the usual absolute value for complex numbers. Given a $\C$-point $z\in X_t$ for $0<|t|\ll 1$, inside some affine chart $U=\text{Spec}(A)$ of $X$, we can define a multiplicative seminorm $A\to \R_{\geq 0}$ (not non-archimedean!) 
\begin{equation}\label{hybridconvergence}
f\mapsto e^{- \log |f(z)|/\log |t|   }=|f(z)|^{1/|\log |t| |}.
\end{equation}
As a sequence of points $z$ move towards $t\to 0$, for any given meromorphic function $f=\sum a_k t^k$ on the base, $\lim_{t\to 0}\log |f(z)|/\log |t|= ord_0(f)$ which is the standard NA valuation on $K$. Thus the points on $X_K^{an}$ are natural limits of the multiplicative seminorms defined by $\C$-points on $X_t$. One can formalize this notion by introducing a \emph{hybrid topology} on $X\sqcup X_K^{an} $, so that $X_K^{an}$ takes the place of the central fibre \cite[Appendix]{Boucksom1}. The functions $f\in A$ then induce local continuous functions on $X\sqcup X_K^{an} $.

The `hybrid' space $X\sqcup \Delta_{\mathcal{X}}$ discussed in section \ref{volumeasymptoteessentialskeleton} can be understood as a finite approximation. Given an snc model $\mathcal{X}$, and take a sequence of $\C$-points $q_k$ tending to $e^{-v}\in X^{an}$, whose image under the retraction map $r_{\mathcal{X} }$ is $x=(x_0,\ldots x_p)\in \Delta_J\subset \Delta_{\mathcal{X}}$.  Tautologically $q_k$ concentrate near $E_J$, and in the local coordinates $z_0,\ldots z_p$, we have $\log |z_i(q_k)|/\log |t|\to v(z_i)= x_i $, which is equivalent to  $\text{Log}_{\mathcal{X} } (z_k)\to x=(x_0,\ldots x_p)\in \Delta_J\subset \Delta_{ \mathcal{X} }$. Formally, the topology on $X\sqcup X_K^{an}$ is the inverse limit of $X\sqcup \Delta_{\mathcal{X}}$ by taking all snc models.


\subsection{Model functions, metrics, positivity}

We now discuss functions, line bundles, and metrics on $X_K^{an}$ \cite{Boucksomsemipositive}.
Given a model $\mathcal{X}$ over $\text{Spec}(R)$ and a Cartier divisor $D$ supported on the central fibre $\mathcal{X}_0$, we can associate a continuous function on $X_K^{an}$ by setting
\[
\phi_D(x)= \max\{  \log |f|_x : f\in \mathcal{O}_{\mathcal{X}  }(D)       \},
\]
The association $D\mapsto \phi_D$ extends by $\Q$-linearity.
Functions obtained in the $\Q$-span using all such choices of models and divisors are called \emph{model functions} on $X_K^{an}$, which form a dense subset of $C^0(X_K^{an})$. The restrictions of such functions to dual intersection complexes are \emph{piecewise affine}.

 To understand the complex geometric meaning, we think of models base changed from snc models over $S$. The divisor $D$ prescribes a class of functions $\phi$ on the total space of the snc model with \emph{analytic singularities}: 
\[
\phi= \log |f|+ C^\infty \text{ function},
\]
where $f$ is a local defining function of $D$. When we consider the rescaling of the restrictions to $X_t$
\[
\phi_t= \frac{1}{ \log |t|  } \phi|_{X_t},
\]
only the singular term is relevant in the limit $t\to 0$, and $\phi_t$ converge to $-\phi_D$ in the hybrid topology.

We think about line bundles on $X_K^{an}$ via the GAGA principle: the line bundles on $X_K^{an}$ correspond to the line bundles $L$ on the scheme $X_K$. A \emph{continuous metric} on $L$ assigns to each local section $s$ a nonnegative continuous  local function $\norm{s}$ on open subsets of $X_K^{an}$, compatible with the sheaf structure, such that $\norm{fs}(x)= |f|_x\norm{s}(x)$, and $\norm{s}>0$ if $s$ is a local frame of $L$.
 Given a continuous metric, any other continuous metric on $L$ is of the form $\norm{\cdot} e^{-\phi}$ for some $\phi\in C^0(X^{an})$, analogous to the usual relation between Hermitian metrics and K\"ahler potentials. As such $\phi$ is referred to as a \emph{potential} function.

Given a model $\mathcal{X}$ for $X_K$, a model $\mathcal{L}$ of $L$ is a line bundle $\mathcal{L}\to \mathcal{X}$ with $\mathcal{L}|_X=L$. To this data we can associate a unique metric $\norm{\cdot}_{ \mathcal{L} }$ on $L$ with the following property: if $s$ is a nowhere vanishing local section of $\mathcal{L}$ on an open set $\mathcal{U}\subset \mathcal{X}$, then $\norm{s}_{ \mathcal{L} } \equiv 1$ on $\mathcal{U}\cap X_K$. This is well defined because any two such sections differ by the multiplication of an invertible function, whose NA absolute value equals the constant one. One can extend this construction to $\Q$-line bundles, and the metrics arising this way are called \emph{model metrics}. They are dense within the continuous metrics.

To see the complex geometric interpretation, we imagine a line bundle $\mathcal{L}$ on some snc models $\mathcal{X}$ over an algebraic curve. Equip $\mathcal{L}$ with any smooth Hermitian metric $h$. Given a local section $s$ of $\mathcal{L}$, the prescription compatible with (\ref{hybridconvergence}) is to consider the local functions on $X_t$ 
\[
z\mapsto  |s(z)|_h^{ 1/|\log |t||  }.
\]
Taking the limit as $t\to 0$, we precisely get the model metric. Notice the ambiguity in the choice of the Hermitian metric is obliterated in the limit.

A paramount notion in K\"ahler geometry is the \emph{positivity} of the metric, usually phrased in terms of psh properties of the potential. The above discussion suggests that in the NA setting, namely $t\to 0$, such a notion should be expressible as a numerical property of the line bundle.

\begin{Def}\cite[Thm. 2.17]{Boucksom}
(Semipositivity I) 
Let $\norm{\cdot}$ be a model metric on $L$, associated to a $\Q$-line bundle $\mathcal{L}$ on a model $\mathcal{X}$ of $X_K$. Then
\begin{itemize}
\item the metric $\norm{ }$ is a semipositive model metric iff $\mathcal{L}$ is nef, namely $\mathcal{L}\cdot C\geq 0$ for any projective curve $C$ contained in $\mathcal{X}_0$;
\item
a continuous metric $\norm{ } e^{-\phi}$ is \emph{semipositive} iff it is the uniform limit of some sequence of semipositive model metrics on $X_K^{an}$.
\end{itemize}
\end{Def}

\begin{rmk}
The advantage of `nef' instead of `ample' is that if we blow up the model further, the pullback of the model line bundle will stay nef, but ampleness will be lost.	
\end{rmk}

\begin{rmk}
Given a continuous metric $\norm{\cdot }$, one can assign to it a `closed (1,1)-form' $\theta$, which for model metrics roughly amounts to taking the numerical class of the model line bundle. This is a formal analogue for the curvature form of a Hermitian metric. For instance, $\norm{\cdot}e^{-\phi}$ is a continuous semipositive metric iff the potential $\phi$ is a continuous $\theta$-psh function. Another theory of forms and currents on Berkovich spaces is developed by Chambert-Loir and Ducros \cite{CLDucros}, which is closer in spirit to differential calculus.
\end{rmk}

In K\"ahler geometry the definition of psh function is local in the complex charts. Since the dual intersection complexes are simplicial objects, one would expect the NA analogous notion to be related to convex functions. This intuition is partially valid:

\begin{prop}\label{convexityonfaces}
\cite[Prop 5.9]{Boucksomsemipositive} Let $\mathcal{X}$ be an snc model for $X_K$, and $\mathcal{L}\to \mathcal{X}$ be a model line bundle for $L\to X$, with associated closed (1,1)-form $\theta$. Then the restriction of any continuous $\theta$-psh function to any face of $\Delta_{\mathcal{X} }\subset X_K^{an}$ is convex.
\end{prop}

The picture is that general $\theta$-psh functions define convex functions on the faces of $\Delta_{ \mathcal{X} }$, and among them the $\theta$-psh model functions give piecewise affine approximations with finer and finer grids.

\begin{rmk}
Gubler and Martin \cite[section 3]{GublerMartin} argue that the a priori global notion of semipositivity can be localized on the Berkovich space, but
their notion is quite abstract. It would be attractive to formulate a notion of convexity for functions on $\Delta_{\mathcal{X} }$ that precisely characterize the restrictions of continuous $\theta$-psh functions from $X_K^{an}$ to $\Delta_{\mathcal{X} }$.
Compare also \cite[Chapter 5]{CLDucros} for another strategy to define plurisubharmonicity on the Berkovich space in terms of positivity of currents.

\end{rmk}

\subsection{NA Monge-Amp\`ere measure}\label{NAMAmeasure}

The NA Monge-Amp\`ere measure \cite{ChambertLoir} is defined through \emph{intersection theory} in a somewhat counterintuitive manner. As a motivation, we consider the complex analytic setting of an snc model $\mathcal{X}$ over an algebraic curve, equipped with a Hermitian line bundle $(\mathcal{L}, h)$ with curvature form $\theta$ in the class $c_1(\mathcal{L})$. Then $\theta^n$ defines a family of $n$-forms on $X_t$, such that $\int_{X_t}\theta^n$ equals the intersection number $(L^n)$. The question is to describe the limit of these $n$-forms, when we view $X_t$ as converging to the dual intersection complex $\Delta_{\mathcal{X} }$ (\cf section \ref{volumeasymptoteessentialskeleton}).

We write $X_0=\sum_{i\in I} b_i E_i$.
Recall that the regions on $X_t$ corresponding to the faces in the dual intersection complex are from the algebraic perspective only small neighbourhoods of $E_J$. Thus the limit of $\theta^n|_{X_t}$ can only be supported at the vertices of $\Delta_{\mathcal{X}}$, which correspond to the components $E_i$. The amount of delta masses concentrated at the vertices are
\[
b_i\int_{E_i} \theta^n= b_i\mathcal{L}^n\cdot E_i,
\]
where $b_i$ appears due to the multiplicity of the sheets. Reassuringly,
\[
\sum_i b_i\mathcal{L}^n\cdot E_i= (L^n)
\]
gives the correct total mass.

Back to the NA setting, 
given a model $\Q$-line bundle $\mathcal{L}\to \mathcal{X}$ for  $L\to X_K$, we write $\mathcal{X}_0= \sum_i b_i E_i$, and denote the divisorial points associated to $E_i$ as $q_i$. We
can then \emph{define} the NA Monge-Amp\`ere measure for the model metric $\norm{\cdot}_ {\mathcal{L} }$ as the following signed atomic measure supported at $q_i\in X_K^{an}$:
\[
MA( \norm{\cdot}_{ \mathcal{L} } )= \sum_{E_i} b_i (\mathcal{L}^n\cdot E_i) \delta_{q_i}
\] 
This definition is compatible with pullback of line bundles by the projection formula, and ensures the total mass is the intersection number $(L^n)$. If $\norm{\cdot}_{\mathcal{L}  }$ is furthermore semipositive, then the intersection numbers are non-negative, so $MA( \norm{\cdot}_{ \mathcal{L} } )$ is a measure. Now a general continuous semipositive metric on $L$ is the uniform limit of a sequence of continuous semipositive model metrics \cite[Cor. 8.8]{Boucksomsemipositive}, and its NA MA measure can be defined as the unique limiting Radon measure of the NA MA measures for the sequence \cite[Cor. 3.5]{Boucksom}.

The theory of NA MA measures bears strong resemblance to the complex pluripotential theory for complex MA measures \cite{Boucksom}\cite{Boucksomsurvey}. The difficulty of working with them lies in the highly nonlocal appearance of the above definition, ultimately caused by the tension between the algebraic limit $\mathcal{X}_0$ and the tropical limit $\Delta_{\mathcal{X}}$. The recent result of Vilsmeier \cite{Vilsmeier} offers a more concrete perspective:

\begin{prop}\label{NAMArealMAcomparisonProp}
(NA MA-real MA comparison) Let $(\mathcal{X},\mathcal{L})$ be a semistable snc model of $(X_K,L)$, and $\text{Int}(\Delta_J)$ be an $n$-dimensional open face of $\Delta_{\mathcal{X} }$. Recall the retraction map $r_{\mathcal{X}}: X_K^{an}\to \Delta_{\mathcal{X} }$. Let $\phi\in C^0(X_K^{an})$ be the potential of a semipositive metric $\norm{\cdot}_{ \mathcal{L} }e^{-\phi}$, and suppose $\phi=\phi\circ r_{\mathcal{X}}$ on $r_{\mathcal{X}}^{-1}(\Delta_J )$,   then on $\text{Int}(\Delta_J)$ the pushforward of the NA MA measure
\[
r_{\mathcal{X}* } MA( \norm{\cdot}e^{-\phi} )= n! MA_\R (\phi|_{\text{Int}(\Delta_J)})
\]
equals the real MA measure of the convex function $\phi|_{\text{Int}(\Delta_J)}$ up to a factor $n!$.
\end{prop}

The rigorous proof of this comparison uses intersection theory. We present a heuristic calculation that hopefully makes the relation between NA MA measure and real MA measure more intuitive to differential geometers. Consider an snc model $\mathcal{X}$ over an algbebraic curve as in the motivation, and assume furthermore that it is semistable. Recall our heuristic dictionary that a metric $\norm{\cdot}$ on $L\to X_K$ should encode a family of Hermitian metrics $h_t$ on $L\to X_t$, such that $h_t^{ 1/|\log |t||  }\to \norm{\cdot}^2$ in the hybrid topology, and the NA MA measure of $\norm{\cdot}$ should be the limit of the measures associated to the curvature forms of $h|_{X_t}$. We now focus on the neighbourhood of an $n$-dimensional open face $\text{Int}(\Delta_J)\subset \Delta_{ \mathcal{X} }\subset  \Delta_{ \mathcal{X} }\sqcup X  $, where we have local coordinates $z_0,\ldots z_n$ with $\prod_0^n z_i=t$, and $x_i=\frac{\log |z_i|}{\log |t|}$. In the local picture we identify metrics with potentials, so $\norm{\cdot}\sim e^{-\phi}$, and after ignoring $C^0$-fluctuation effects $h_t^{  1/|\log |t||}\sim e^{-  2\phi\circ \text{Log}_{\mathcal{X}} }$.
Imposing more smoothness assumptions, the curvature form of $h_t$ is approximately
\[
 |\log |t|| dd^c\phi\circ\text{Log}_{ \mathcal{X}}= \frac{-1}{2\pi} \sum_{1\leq i,j\leq n} \frac{\partial^2\phi}{\partial x_i \partial x_j} dx_i \wedge d\text{arg}(z_j).
\]
The NA MA measure should agree with the limiting pushforward measure
\[
\lim_{t\to 0}\text{Log}_{ \mathcal{X}*} (|\log |t|| dd^c\phi\circ\text{Log}_{ \mathcal{X}} )^n = n!\det(D^2 \phi) |dx_1\ldots dx_n|=n! \text{MA}_\R(\phi)
\]
which equals the real MA measure up to the factor $n!$.

\begin{rmk}
In this heuristic calculation, the assumption for $\phi$ to factor through the retraction map allows us to replace the hybrid space $X\sqcup X_K^{an}$ by its finite approximation $X\sqcup \Delta_{\mathcal{X} }$.
\end{rmk}

\begin{rmk}
The above calculation is similar to the formalism developed by Chambert-Loir and Ducros, see \cite[Lemma 5.7.1]{CLDucros}.
\end{rmk}

\subsection{NA Calabi conjecture}\label{NACalabi}

The central result of NA pluripotential theory is the solution to the NA analogue of the Calabi conjecture. A good survey is \cite{Boucksomsurvey}.

\begin{thm}
\cite{Boucksom} Let $X_K$ be a smooth projective K-scheme arising from the base change of an algebraic degeneration family. Let $L$ be an ample line bundle on $X_K$, and $d\mu$ be a Radon probability measure supported on the dual intersection complex of some snc model of $X_K$. Then there is a unique continuous semipositive metric $\norm{\cdot}$ on $L$, such that
\[
MA( \norm{\cdot })= (L^n) d\mu.
\]

\end{thm}

Their strategy uses a \emph{variational method}. There is a concave \emph{energy functional} $\mathcal{E}$ on the space of continuous semipositive metrics on $L$ (equivalently viewed as continuous $\theta$-psh potentials $\phi$), whose first variation is given by the NA MA measure. One seeks a maximizer of the functional
\[
F_\mu(\phi)= \mathcal{E}(\phi)-(L^n)\int_{X_K^{an}} \phi d\mu,
\]
by first enlarging the space of $\phi$ to a function space $PSH(X_K^{an}, \theta)$ which is \emph{compact} modulo the addition of a real constant; this is analogous to the $L^1$-compactness of $PSH(X,\omega)/\R$ in the K\"ahler setting. The notions of the NA MA measure and the energy functional extend naturally to the \emph{energy class} functions inside $PSH(X_K^{an}, \theta)$, much like in the complex pluripotential theory setting. One then shows the maximizer is in fact a \emph{critical point}, namely a weak solution to the NA MA equation. This is subtle since small perturbations of functions in $PSH(X_K^{an}, \theta)$ may fall outside of the class by losing positivity. One proves the continuity of the weak solution using analogues of Kolodziej's estimates. The uniqueness of the solution again relies on the concavity of $\mathcal{E}$.

While this strategy shares a very similar logical structure with the complex analytic setting, the technical foundations are built upon intersection theory and vanishing theorems in birational geometry, instead of differential operators.

The main case of interest to us is when $X_K$ arises from polarized algebraic maximal degeneration of CY manifolds $L\to X$. Then NA pluripotential theory provides a solution to
\begin{equation}\label{NAMACY}
MA( \norm{\cdot }_{CY})= (L^n) d\mu_0,
\end{equation}
where $d\mu_0$ is the Lebesgue measure supported on the essential skeleton $Sk(X)\subset X_K^{an}$ (\cf section \ref{volumeasymptoteessentialskeleton}).

\begin{Def}(\cf section \ref{NAMAmeasure})
We say $\norm{\cdot }_{CY}$ satisfies the \emph{NA MA-real MA comparison property}, if there exists a semistable snc model $(\mathcal{X}, \mathcal{L})$ of $(X,L)$ with the property that, the potential $\phi_0$ defined by $\norm{\cdot}_{ CY }= \norm{\cdot}_{ \mathcal{L}}e^{-\phi_0}$ satisfies $\phi_0=\phi_0\circ r_{\mathcal{X}  }$ on the preimages of the retraction map over all the $n$-dimensional open faces 
$\text{Int}(\Delta_J)\subset Sk(X)$.

\end{Def}

Notice $\text{Int}(\Delta_J)\subset \Delta_{\mathcal{X} }$ inherits a natural integral affine structures.  
Since the restriction of $\phi_0$ is convex on these faces by Prop. \ref{convexityonfaces}, its real MA measure makes sense, and by Prop. \ref{NAMArealMAcomparisonProp} it satisfies the \emph{real MA equation} on $\text{Int}(\Delta_J)$
\begin{equation}\label{realMACY}
\text{MA}_\R(\phi_0)= \frac{ (L^n) }{ n!} d\mu_0.
\end{equation}
Then the regularity theory of real MA equation (\cf section \ref{RegularitytheoryforrealMA}) will apply, so we may view the comparison property as a regularity assumption on $\norm{\cdot}_{CY}$. Some subtleties are discussed in \cite[Appendix]{GublerJill}.


\subsection{Approximation by Fubini-Study metrics}

A fundamental result in K\"ahler geometry is that any K\"ahler metric in an integral class can be approximated by Fubini-Study metrics associated with projective embeddings. While the usual Fubini-Study metric depends on a choice of a Hermitian inner product on the $\C$-vector space of global sections, the NA analgoue depends on a NA norm on the $K$-vector space $V=H^0(X_K,m L)$ for $m\gg 1$, with the ultrametric property $\norm{x+y}_V\leq \max\{ \norm{x}_V, \norm{y}_V \}$. In our case $K=\C(\!(t)\!)$ one can select a $K$-basis $s_0, s_1, \ldots s_N$ for $V$ (called an `orthogonal basis' \cite[section 2.1.2]{Fang}), such that
\[
\norm{ a_0s_0+ \ldots+ a_N s_N}_V=\max\{ |a_0|\norm{s_0}_V,\ldots ,|a_N|  \norm{s_N}_V  \}, \quad \forall a_i\in K.
\]
The NA Fubini-Study metric on $L\to X_K$ can be defined as
\[
\norm{s}_{FS}(x)=  \inf_{ \tilde{s}\in V, \tilde{s}(x)=s^{\otimes m}(x)} \norm{ \tilde{s}}_V^{1/m}, \quad \forall x\in X_K^{an}.
\]
Concretely in the orthogonal basis, written in a local trivialisation,
\begin{equation}\label{FubiniStudyNA}
\norm{s}_{FS }(x)= \frac{ |s(x)|}{ \max_j \{  |s_j(x)|/\norm{s_j}_V  \}^{1/m}   }, \quad \forall x\in X^{an}.
\end{equation}
A NA analogue of the \emph{Fubini-Study approximation theorem} gives an alternative view on semipositive metrics:

\begin{prop}
(Semipositivity II)
\cite{ChenMoriwaki} Assume $L\to X_K$ is ample. Then a continuous metric on $L$ is semipositive iff it can be written as a uniform limit of Fubini-Study metrics.
\end{prop}

\begin{rmk}
Given an NA norm $\norm{\cdot}_V$ on the $K$-vector space $V=\C^{N+1}\otimes_\C K$, finding an orthogonal basis in general requires access to formal Laurent series. If we only use sections which are finite Laurent polynomials in $t$, then for any given $\epsilon>0$, we can find a $K$-basis $s_0, \ldots s_N$ such that
$
\norm{ a_0s_0+ \ldots+ a_N s_N}_V'=\max\{ |a_0|\norm{s_0}_V,\ldots ,|a_N| \norm{s_N}_V   \}$ satisfies
\[
(1-\epsilon) \norm{\cdot}_V'\leq \norm{\cdot}_V \leq \norm{\cdot}_V',
\]
by \cite[Prop. 1.3]{ChenMoriwaki}. The upshot is that in the approximation theorem we may assume $s_i$ to be finite Laurent polynomials.
\end{rmk}

For the complex geometric interpretation, we assume as usual $X_K$ is the base change of an algebraic degeneration family $X$, with an ample polarisation line bundle $L$. For any given NA Fubini-Study metric (\ref{FubiniStudyNA}), we can associate a family of Fubini-Study metrics on $(X_t, L)$: 
\begin{equation}\label{FubiniStudyXt}
\norm{s}_{FS,t}(z)=  \frac{ |s(z)|}{ \{ \{ \sum_j  |s_j(z,t) |^2 |t|^{2\log \norm{s_j}_V}  \}^{1/2m}   }, \quad \forall z\in X_t.
\end{equation}
Here $s_j$ make sense for finite $t$ because they are selected as finite Laurent polynomials in $t$. By our dictionary, we should consider the limit of $\norm{s}_{FS,t}^{1/|\log |t||  }$ as $t\to 0$. Since
\[
0\leq \log (\sum_j |s_j|^2 |t|^{2\log {\norm{s_j}_V} }   )^{1/2} -\log \max_j |s_j| |t|^{\log \norm{s_j}_V} \leq \log (N+1),
\]
in the limit the difference between maximum and square length disappears, so $\norm{s}_{FS,t}^{1/|\log |t||  }$ converges to (\ref{FubiniStudyNA}) in the hybrid topology on $X\sqcup X_K^{an}$.


\section{Potential estimates and SYZ fibration}\label{PotentialestimateSYZfibration}

We consider a polarized algebraic maximal degeneration of Calabi-Yau manifolds $X\to S\setminus \{0\}$ over a smooth algebraic curve. Let $(\mathcal{X}, \mathcal{L})$ be a semistable snc model with $\mathcal{L}|_X=L$. The NA pluripotential theory provides a continuous semipositive metric $\norm{\cdot}_{CY}=\norm{\cdot}_{\mathcal{L} }e^{-\phi_0}$ on $L$ over $X_K^{an}$ solving the NA MA equation (\ref{NAMACY}), which we assume henceforth satisfies the NA MA-real MA comparison property, so $\phi_0$ solves the real MA equation over the $n$-dimensional open faces $\text{Int}(\Delta_J)$ of the essential skeleton $Sk(\mathcal{X} )$ (\cf section \ref{NACalabi}).

\subsection{Comparison K\"ahler metric I}

We apply the Fubini-Study approximation to transfer the NA metric $\norm{\cdot}_{CY}$ to the complex manifolds $X_t$ up to $C^0$-small errors in the potential, while \emph{preserving the positivity} of the metric.

Recall there is a logarithm map $\text{Log}_{ \mathcal{X} }: X_t\to \Delta_{\mathcal{X} }$, defined up to $O( \frac{1}{ |\log |t|| } )$ coordinate ambiguity, so when we refer to $\text{Log}_{ \mathcal{X} }$ over $\text{Int}(\Delta_J)$ we will implicitly shrink $\text{Int}(\Delta_J)$ by $O(\frac{1}{ |\log |t|| }   )$ to ensure the coordinate expression $x_i= \frac{ \log |z_i|}{ \log |t| }$ makes sense.

\begin{lem}\label{comparisonmetricIlemma}
Given any $0<\epsilon\ll 1$, then for sufficiently small $t$ depending on $\epsilon$, there is a smooth K\"ahler metric $\omega_{FS,t}$ on $(X_t, \frac{1}{|\log |t|| }c_1(L))$, such that 
\begin{itemize}
\item The relative K\"ahler potential for any two choices of $\epsilon$ is bounded uniformly independent of small $t$.
\item On $\text{Log}_{ \mathcal{X} }^{-1}(\text{Int}(\Delta_J))$, the local K\"ahler potentials 
$\phi_{J,t}$ of $\omega_{FS,t}$ can be chosen to  satisfy $|\phi_{J,t}-  \phi_0\circ \text{Log}_{ \mathcal{X}}|<\epsilon$.

\end{itemize}

\end{lem}

\begin{proof}
Let $\norm{\cdot}_{FS}$ be a NA Fubini-Study approximation of $\norm{\cdot}_{CY}$, with potential difference less than $\epsilon/2$. 
We construct the Fubini-Study metrics $\norm{\cdot }_{FS,t}$ on $(X_t, L)$ by the formula 
(\ref{FubiniStudyXt}), so the curvature forms of $\norm{\cdot }_{FS,t}^{1/|\log |t||}$ define the K\"ahler metrics  $\omega_{FS,t}$ on $(X_t, \frac{1}{|\log |t|| }c_1(L))$. By construction, for sufficiently small $t$ the local potentials are $C^0$-close to that of $\norm{\cdot}_{FS}$, which is $\epsilon$-close to the continuous metric $\norm{\cdot}_{CY}$, so the uniform boundedness of the potentials can be guaranteed.

By our NA MA-real MA comparison assumption, over $\text{Int}(\Delta_J)$ the potential $\phi_0$ of $\norm{\cdot}_{CY}$ equals the pullback $\phi_0\circ r_{\mathcal{X} }$  via the retraction map $r_{\mathcal{X} }$. From our discussions on the hybrid topology in section \ref{Berkovichspace}, over $\text{Int}(\Delta_J)$, for $\phi_{J,t}$ to be $C^0$-close to $\phi_0\circ r_{\mathcal{X} }$ means the same as saying $\phi_{J,t}$ is $C^0$-close to $\phi_0\circ \text{Log}_{\mathcal{X} }$.

\end{proof}

\subsection{Comparison K\"ahler metric II: Regularisation}

We now improve the metric $\omega_{FS,t}$ on $(X_t, \frac{1}{|\log |t|| }c_1(L))$ to make the volume form approximately CY except on a set with a small percentage of the CY measure.

Let $0<\delta\ll 1$. Since $\phi_0$ solves the real MA equation (\ref{realMACY}) on the $n$-dimensional open faces $\text{Int}(\Delta_J)$, the regularity theory of real MA surveyed in section \ref{RegularitytheoryforrealMA} applies. In particular, we can find finitely many small open balls $B_l=B(p_l, 2r(\delta))$ properly contained in $ \cup_J\text{Int}(\Delta_J)$, such that $\phi_0$ has $C^k$-norm uniformly bounded by $C(\delta)$ on any $B_l$, and the complement of $W_\delta=\cup_l B(p_l, r(\delta))$ has small $\mu_0$-measure less than $\delta$. The open sets $W_\delta$ form an exhaustion of $W_0=\cup_{\delta>0} W_\delta$, which is the $C^\infty$-locus of the real MA solution $\phi_0$. Recall the normalised CY measure $d\mu_t$ from section \ref{volumeasymptoteessentialskeleton}.

\begin{lem}(Regularisation)\label{regularisationlemma}
Let $\epsilon$ be suffiently small dependent on $\delta$, and construct $\omega_{FS,t}$ as in Lemma \ref{comparisonmetricIlemma}, for $t$ sufficiently small dependent on $\epsilon$ and $\delta$.		
There is a Lipschitz continuous function $\psi_t$ on $X_t$ with $\norm{\psi_t}_{L^\infty}\leq 3\epsilon$, such that
\begin{itemize}
\item The function $\psi_t$ is smooth away from a closed subset with $d\mu_t$-measure zero.
\item 
The (1,1)-current $\omega_{\psi,t}=\omega_{FS,t}+dd^c\psi_t\geq 0$ is positive on $X_t$.

\item The metric estimate $(1-C\epsilon)dd^c \phi_0\circ \text{Log}_{\mathcal{X} }\leq \omega_{\psi,t} \leq (1+C\epsilon) dd^c \phi_0\circ \text{Log}_{\mathcal{X} }$ holds on the smooth locus of $\psi$ inside $\text{Log}_{\mathcal{X} }^{-1}(W_\delta)$.

\item The total variation $ \int_{X_t} | \frac{ |\log |t||^n \omega_{\psi,t}^n }{ (L^n) }  -d\mu_t|  <\delta.$
\end{itemize}

\end{lem}

\begin{proof}
Choose a smooth nonnegative bump function $\eta$ on $\R^n$ supported in $B(0,2)$, and equals one on $B(0,1)$, and construct $\eta_i=\eta( \frac{|x-p_i|}{r(\delta)} )$ supported on $B_i$. We calculate $\text{Log}_{\mathcal{X} }^{-1}( B_l)$
\[
\begin{split}
dd^c \phi_0\circ \text{Log}_{\mathcal{X} }= \frac{1}{4\pi|\log |t||^2}  \sum_{1\leq i,j\leq n} \frac{\partial^2 \phi_0}{\partial x_i\partial x_j} \sqrt{-1} d\log z_i\wedge d\log \bar{z}_j 
\\
\gtrsim_\delta  \frac{1}{ |\log |t||^2 } \sum  \sqrt{-1} d\log z_i\wedge d\log \bar{z}_j,
\end{split}
\]
\[
dd^c \eta_l\circ \text{Log}_{\mathcal{X} }\geq -\frac{C}{ |\log |t||^2 } \sum  \sqrt{-1} d\log z_i\wedge d\log \bar{z}_j,
\]
so if $0<\epsilon\ll 1$ is sufficiently small dependent on $\delta$, we can ensure $(\phi_0+2\epsilon \eta_l)\circ \text{Log}_{\mathcal{X} }$ is psh on $\text{Log}_{\mathcal{X} }^{-1}( B_l)$.

Consider the potential function on $\text{Log}_{\mathcal{X} }^{-1}(\text{Int}(\Delta_J))$
\[
\psi_t(z)=\max(0, \max_{ \text{Log}_{\mathcal{X}}(z)\in B_l  } \{  (\phi_0+2\epsilon \eta_l)\circ \text{Log}_{\mathcal{X} }- \epsilon -\phi_{J,t}          \} ),
\]
with $\phi_{J,t}$ from Lemma \ref{comparisonmetricIlemma}.

Since $|\phi_{J,t}- \phi_0\circ \text{Log}_{ \mathcal{X}}|<\epsilon$, we see that on $\partial B_l$ the maximum is never achieved by $(\phi_0+2\epsilon \eta_l)\circ \text{Log}_{\mathcal{X} }- \epsilon -\phi_{J,t}  $, so the fact that this term is only locally defined causes no problem. By construction every term is $dd^c\phi_{J,t}$-psh, so the maximum $\psi_t$ is also  $dd^c\phi_{J,t}$-psh.  Near the boundary of  $\text{Int}(\Delta_J)$ the maximum is achieved by $\psi_t=0$, so $\psi_t$ globalizes to define a Lipschitz continuous $\omega_{FS,t}$-psh function on $X_t$, with $\norm{\psi_t}_{L^\infty}\leq 3\epsilon$.

Morever, on $\text{Log}_{\mathcal{X} }^{-1}(W_\delta)$ the maximum of $\psi_t$ is strictly greater than zero. Perturbing the bump function in the construction if necessary, we may assume the locus on $X_t$ where the maximum is achieved by at least two terms is a subset of codimension one, and it is automatically closed. So a.e on $\text{Log}_{\mathcal{X} }^{-1}(W_\delta)$,
the metric $\omega_{\psi,t}=\omega_{FS,t}+ dd^c\psi_t$ is smooth and equals $dd^c(\phi_0+2\epsilon \eta_l)\circ \text{Log}_{\mathcal{X} } $ for some $l$. We calculate using the regularity estimates on $dd^c\phi_0$ that
\[
\omega_{\psi,t}^n=(dd^c(\phi_0+2\epsilon \eta_l)\circ \text{Log}_{\mathcal{X} } )^n=  ( 1+ O(\epsilon) ) (dd^c\phi_0\circ \text{Log}_{\mathcal{X} } )^n.
\]
Using the real MA equation (\ref{realMACY}), 
\[
(dd^c\phi_0\circ \text{Log}_{\mathcal{X} } )^n= n!\det(D^2\phi_0) \prod_i \frac{1}{4\pi|\log |t||^2} \sqrt{-1} d\log z_i\wedge d\log \bar{z}_i  ,
\]
where $\frac{(L^n)}{n!}d\mu_0=\det(D^2 \phi_0) dx_1\ldots dx_n$ determines the constant $\det(D^2 \phi_0)$. Comparing with section \ref{volumeasymptoteessentialskeleton}, the normalized CY measure $d\mu_t$ satisfies  
\[
\frac{(L^n)}{n!}d\mu_t= (1+ O(\frac{1}{|\log |t|| } ) ) \det(D^2 \phi_0) \prod_i \frac{1}{4\pi|\log |t||} \sqrt{-1} d\log z_i\wedge d\log \bar{z}_i .
\]
For sufficiently small $t$ depending on $\epsilon$ and $\delta$, we combine the above to deduce
\[
\omega_{\psi,t}^n= (1+O(\epsilon)) \frac{(L^n) }{ |\log |t||^n }d\mu_t.
\]
The total complex MA measure is $\int_{X_t}\omega_{\psi,t}^n= \frac{(L^n)}{|\log |t||^n}  $, and the contribution from the smooth region in  $\text{Log}_{\mathcal{X} }^{-1}(W_\delta)$ is greater than
$
(1-\delta) \frac{(L^n)}{|\log |t||^n} 
$ for very small $\epsilon$ dependent on $\delta$. Thus the measure contribution from the complement must be less than $\delta \frac{(L^n)}{|\log |t||^n} 
$, namely $\delta$-percent of the total measure. Thus the total variation of the signed measure $d\mu_t- \frac{ |\log |t||^n }{ (L^n) } \omega_{\psi,t}^n$ is smaller than $\delta
$ for small enough $\epsilon$. 
\end{proof}

\subsection{Potential estimate I}

We denote the CY metrics on $(X_t, \frac{1}{|\log |t||}c_1(L))$ as
\[
\omega_{CY,t}=\omega_{FS,t}+ dd^c \phi_{CY,t}.
\]
Using the local potentials $\phi_{J,t}$ of $\omega_{FS,t}$, we can write $\omega_{CY,t}$ in terms of local absolute potentials on $\text{Log}_{\mathcal{X} }^{-1}(\text{Int}(\Delta_J))$:
\[
\omega_{CY,t}=dd^c \phi_{CY,J,t}, \quad \phi_{CY,J,t}=\phi_{J,t}+ \phi_{CY,t}.
\]
This depends on an implicit choice of $\omega_{FS,t}$ and $\phi_{J,t}$ in Lemma 
\ref{comparisonmetricIlemma}. Our goal is to find a suitable choice and show the smallness of $|\phi_{CY,J,t}- \phi_0\circ\text{Log}_{\mathcal{X} }| $ in the generic region.

\begin{prop}\label{UniformLinftyCY}
For $0<|t|\ll 1$, the CY potentials have a uniform bound $\norm{\phi_{CY,t}}_{L^\infty} \leq C$ under the normalisation $\sup_{X_t}\phi_{CY,t}=0$.
\end{prop}

\begin{proof}
Combine the uniform $C^0$-estimate Theorem \ref{pluripotentialthm1} and the uniform Skoda estimate Theorem \ref{UniformSkodathm}, we know the CY potential with respect to any fixed choice of Fubini-Study background metric is uniformly bounded for small $t$. Here the implicit choice of $\epsilon$ in $\omega_{FS,t}$ does not matter because of Lemma \ref{comparisonmetricIlemma}.
\end{proof}

\begin{prop}
Given small numbers $0<\lambda, \kappa\ll 1$, then for $\delta, \epsilon, t$ sufficiently small depending on $\lambda$ and $\kappa$, the function $\phi_{CY,t}$ is near its minimum with large probability:
\[
d\mu_t (\{   \phi_{CY,t}- \min_{X_t} \phi_{CY,t} \geq \kappa/5    \} ) < \lambda. 
\]
\end{prop}

\begin{proof}
We wish to compare $\omega_{CY,t}$ with $\omega_{\psi,t}$ from Lemma \ref{regularisationlemma} by an $L^1$-stability estimate. 
Pick a parameter $c$ such that
\[
d\mu_t(  \{ \psi_t- \phi_{CY,t} -c\leq 0 \}      )\geq  \lambda.
\]
Since the potential $\phi_{CY,t}$ has a uniform bound, Theorem \ref{UniformSkodathm} implies another uniform Skoda estimate with modified constants
\[
\int_{X_t} e^{-\alpha u}  d\mu_t \leq A, \quad \forall u\in PSH(X_t, \omega_{CY,t}) \text{ with } \sup_{X_t} u=0.
\]
We also have the $L^1$-stability property for $\omega_{\psi,t}$ in Lemma \ref{regularisationlemma}:
\[
\int_{X_t} |d\mu_t- \frac{ \omega_{\psi,t}^n}{ \text{Vol}(X_t, \omega_{\psi,t}) }|< \delta.
\]

We then apply the uniform $L^1$-stability estimate Theorem \ref{UniformL1stabilitythm}, 
with $Y=X_t$, $\omega= \omega_{CY,t}$ and $\phi=\psi_t-\phi_{CY,t}-c$. In this construction $\delta, \epsilon, t$ are sufficiently small, chosen successively depending on $\lambda$ and $\kappa$. We conclude
\[
\sup_{X_t} (\psi_t-\phi_{CY,t}-c) \leq C(\lambda)\delta^{1/(2n+3)} \ll \kappa.
\]
Now $|\psi_t|\leq 3\epsilon \ll \kappa$, so 
$
\inf_{X_t} \phi_{CY,t} > -c-\kappa/10.
$

Taking the contrapositive, if we choose $c= -\inf_{X_t} \phi_{CY,t}-\kappa/10 $, then 
\[
d\mu_t(  \{ \psi_t- \phi_{CY,t} -c\leq 0 \}      )< \lambda,
\]
whence for $\epsilon\ll \kappa$, using again $|\psi_t|\leq 3\epsilon$,
\[
d\mu_t(  \{ \phi_{CY,t} - \inf \phi_{CY,t}\geq \kappa/5 \}      )< \lambda.
\]
\end{proof}

\begin{rmk}
The reason we use an asymmetric version of the $L^1$-stability estimate, is that we have no control on the density of the comparison metric $\omega_{\psi,t}$ away from the the generic region except for a small bound on the measure contribution there.
\end{rmk}

We can reformulate this in terms of the local potentials of the CY metrics, and thereby eliminate auxiliary choices of Fubini-Study metric and regularisation.

\begin{cor}
Given small numbers $0<\lambda, \kappa\ll 1$, then for $t$ small enough depending on $\lambda, \kappa$, there exist appropriately chosen local potentials $\phi_{CY,J,t}$ on $\text{Log}_{\mathcal{X} }^{-1}(\text{Int}(\Delta_J) )$, normalized to $\inf (\phi_{CY,J,t}- \phi_0\circ \text{Log}_{\mathcal{X} }) =0$, satisfying
\[
d\mu_t( \{  \phi_{CY,J,t}- \phi_0\circ \text{Log}_{\mathcal{X} }\geq \kappa/4 \} ) <\lambda,
\]
and $\norm{ \phi_{CY,J,t}- \phi_0\circ \text{Log}_{\mathcal{X} } }_{L^\infty} \leq C$ independent of $\lambda, \kappa$ and small $t$.

\end{cor}

\begin{proof}
By construction in Lemma \ref{comparisonmetricIlemma}, the local potential $\phi_{J,t}$ of $\omega_{FS,t}$ is $\epsilon$-close to $\phi_0\circ \text{Log}_{\mathcal{X} }$, and since $\epsilon\ll \kappa$ these two are practically the same. Up to an overall normalisation constant, which is fixed by $\inf=0$, we have
$
\phi_{CY,J,t}= \phi_{CY,t}+ \phi_{J,t},
$
so the measure bound follows from the previous result.

The uniform $L^\infty$ bound follows from Prop. \ref{UniformLinftyCY} and Lemma \ref{comparisonmetricIlemma} without any reference to $\lambda, \kappa$.
\end{proof}

Given $0<\tau\ll 1$, we consider the region obtained from shrinking the $n$-dimensional faces near the boundary:
\[
U_{J,t,\tau}= \text{Log}_{\mathcal{X}}^{-1}(  \Delta_J\setminus  \{ |x_i|<\tau, |1-\sum_1^n x_i|<\tau   \} ).
\]
The following theorem is a precise formulation for $C^0_{loc}$-convergence of the local CY potentials to $\phi_0$ over the $n$-dimensional open faces of $Sk(X)$ as $t\to 0$.

\begin{thm}\label{C0convergence}
($C^0_{loc}$-convergence estimate on the potential)
Given $0<\tau, \kappa\ll 1$, then for sufficiently small $t$, on each $U_{J,t,\tau}$ there is a $C^0$-bound
\[
0\leq \phi_{CY,J,t}- \phi_0\circ \text{Log}_{\mathcal{X} }<\kappa.
\] 
\end{thm}

\begin{proof}
We need to obtain upper bound on $\phi_{CY,J,t}$. Consider $z\in U_{J,t,\tau}$ and $x=\text{Log}_{\mathcal{X} }(z)$.
Let $r\ll \tau$ be a parameter to be fixed, so $B(x,3r)\subset \text{Int}(\Delta_J)$. The function $\phi_0$ has an a priori Lipschitz estimate on $\Delta_{\mathcal{X}}$, so the oscillation of $\phi_0$ on $B(x,3r)$ is less than $Cr\ll \kappa$ by choosing $r$ small enough.

Now we apply the mean value inequality to the psh function $\phi_{CY,J,t}$ on a ball in the local covering space of $U_{J,t,\tau}\subset (\C^*)^n$, which projects to $B(x,r)$ via $\text{Log}_{\mathcal{X} }$. We have
\[
\phi_{CY,J,t}(z) \leq \dashint_{ball} \phi_{CY,J,t},
\]
hence
\[
\begin{split}
(\phi_{CY,J,t} - \phi_0\circ \text{Log}_{\mathcal{X} })(z) & \leq \text{osc}_{B(r)} \phi_0 + \dashint_{ball} (\phi_{CY,J,t}- \phi_0\circ \text{Log}_{\mathcal{X} })\\
& \leq \frac{\kappa}{10}+ \dashint_{ball} (\phi_{CY,J,t}- \text{Log}_{\mathcal{X} }).
\end{split}
\]
But on the ball $\phi_{CY,J,t}- \text{Log}_{\mathcal{X} }< \kappa/4$, except on a subset of the ball with $d\mu_t$-percentage $\leq C\lambda r^{-n}  $, on which we use the coarser bound 
 $\phi_{CY,J,t}- \phi_0\circ \text{Log}_{\mathcal{X} } \leq C$.
Combining these,
\[
(\phi_{CY,J,t} - \phi_0\circ \text{Log}_{\mathcal{X} })(z) < \kappa/2+ C\lambda r^{-n} <\kappa,
\]
by choosing $\lambda$ sufficiently small depending on $\kappa, \tau$.
\end{proof}

\begin{rmk}
The above estimates do not use the full strength of the regularisation lemma \ref{regularisationlemma}. We only use the $L^1$-stability of the volume density, not the metric information.
\end{rmk}

\subsection{Potential estimate II}

Here we present a  second strategy for the potential estimate, which aims to circumvent the uniform $L^1$-stability estimate Theorem \ref{UniformL1stabilitythm}, and we explain why there is a difficulty with this second approach. Readers who wish to follow the main line of the proof may skip this section.

\begin{lem}
Given $0<\delta\ll 1$, then for $0<\epsilon\ll 1$ depending on $\delta$, and $t$ small enough depending on $\epsilon, \delta$, 
\[
\int_{ \text{Log}_{\mathcal{X} }^{-1}(W_\delta) } d(\psi_t-\phi_{CY,t})\wedge d^c(\psi_t-\phi_{CY,t}) \wedge (dd^c\phi_0\circ\text{Log}_{\mathcal{X} })^{n-1} \leq \frac{C\delta}{ |\log |t||^n} , 
\]	
where $C$ is independent of $\delta, \epsilon, t$.
\end{lem}

\begin{proof}
Pretending everything is smooth, a standard integration by part gives
\[
\begin{split}
&\int_{X_t}(\psi_t-\phi_{CY,t})(\omega_{CY,t}^n- \omega_{\psi,t}^n) \\
=& \int_{X_t} (\psi_t-\phi_{CY,t}) dd^c (-\psi_t+\phi_{CY,t}  ) \wedge (\omega_{CY,t}^{n-1}+\ldots+ \omega_{\psi,t}^{n-1}  ) \\
=& \int_{X_t} d(\psi_t-\phi_{CY,t})\wedge d^c(\psi_t-\phi_{CY,t}) \wedge (\omega_{CY,t}^{n-1}+\ldots+ \omega_{\psi,t}^{n-1}  ) \\
\geq &  \int_{X_t} d(\psi_t-\phi_{CY,t})\wedge d^c(\psi_t-\phi_{CY,t}) \wedge  \omega_{\psi,t}^{n-1} . 
\end{split}
\]
The same calculations work for continuous $\omega_{FS,t}$-psh functions by standard pluripotential theory.

Combine $\norm{\phi_{CY,t} }_{L^\infty }\leq C$ with the total variation bound in Lemma \ref{regularisationlemma},
\[
 \int_{X_t} | \frac{ |\log |t||^n \omega_{\psi,t}^n }{ (L^n)  }  -d\mu_t|  <\delta,
\]
we get
\[
\int_{X_t} d(\psi_t-\phi_{CY,t})\wedge d^c(\psi_t-\phi_{CY,t}) \wedge  \omega_{\psi,t}^{n-1} \leq C\int_{X_t}|\omega_{CY,t}^n- \omega_{\psi,t}^n| \leq \frac{C\delta}{ |\log |t||^n} . 
\]
Again by Lemma \ref{regularisationlemma}, the metric $\omega_{\psi,t}$ is uniformly controlled a.e. on $\text{Log}_{\mathcal{X} }^{-1}(W_\delta)$, so
\[
\int_{ \text{Log}_{\mathcal{X} }^{-1}(W_\delta) } d(\psi_t-\phi_{CY,t})\wedge d^c(\psi_t-\phi_{CY,t}) \wedge (dd^c\phi_0\circ\text{Log}_{\mathcal{X} })^{n-1} \leq \frac{C\delta}{ |\log |t||^n} . 
\]
\end{proof}

An outline of this strategy is
\begin{itemize}
\item Choose some suitable integral normalisation on $\phi_{CY,t}$. Apply Poincar\'e inequality to prove the average $L^2$-integral of $\psi_t-\phi_{CY,t}$ is small in the generic region $W_\delta$, which occupies most of the $d\mu_t$-measure.

\item
Deduce the measure is small on the set where $\phi_{CY,t}-\psi_t$ is perceptibly negative.

\item
Apply the stability estimate Cor. \ref{StabilityestimateKolodziej} to show $\phi_{CY,t}-\psi_t$ cannot be perceptibly negative. Since $\norm{\psi_t}_{C^0}$ is negligible, it shows the minimum of $\phi_{CY,t}$ is almost zero.

\item
Using the small $L^2$-average bound on $\phi_{CY,t}$, one applies the mean value inequality to derive a small upper bound on $\phi_{CY,t}$ in the generic region $W_\delta$.

\end{itemize}

The problem lies in the fact that the $n$-dimensional open faces $\text{Int}(\Delta_J)$ of $Sk(X)$ are disconnected, so the average values on each face are a priori unrelated. Thus the Poincar\'e inequality can only imply a small bound on $\phi_{CY,t}-\psi_t$ with a priori different normalisations associated to each open face, which is not good enough to get a small bound on average $L^2$-integral of $\psi_t-\phi_{CY,t}$.







\subsection{Metric convergence and SYZ fibration}

Given the $C^0_{loc}$-convergence estimate Theorem \ref{C0convergence}, then the metric SYZ conjecture would follow as explained in \cite{LiFermat}. The most important step is the following $C^\infty_{loc}$-convergence result in the generic region. Recall the exhaustion $W_\delta$ for the regular locus of the real MA solution $\phi_0$. The open subsets 
$\text{Log}_{\mathcal{X} }^{-1}(W_\delta)$ occupy almost the full percentage of the $d\mu_t$-measure on $X_t$ for small $\delta, t$, and as such deserve the name `generic region'.

\begin{thm}
(Metric $C^\infty_{loc}$-convergence in the generic region) For any given $0<\delta\ll 1$, then as $t\to 0$,
\[
\norm{ \phi_{CY,J,t} - \phi_0\circ \text{Log}_{\mathcal{X} } }_{C^k( \text{Log}_{\mathcal{X} }^{-1}(W_\delta)  )} \to 0,
\]
where the $C^k$-norm is defined by passing to the local universal cover of $\text{Log}_{\mathcal{X} }^{-1}(W_\delta)\subset (\C^*)^n$ with preferred coordinates $\zeta_i=\frac{1}{\log |t|} \log z_i$ for $i=1,2,\ldots, n$.
\end{thm}

\begin{proof}
By the calculations in the proof of Lemma \ref{regularisationlemma},
\[
(dd^c\phi_0\circ \text{Log}_{\mathcal{X} } )^n= n!\det(D^2\phi_0) \prod_i \frac{1}{4\pi} \sqrt{-1} d\zeta_i\wedge d\bar{\zeta}_i ,
\]
while the CY condition gives (\cf section \ref{volumeasymptoteessentialskeleton})
\[
\begin{split}
&(dd^c \phi_{CY,J,t})^n = \omega_{CY,t}^n= \frac{ (L^n) }{ |\log |t||^n } d\mu_t=  \frac{ (L^n) }{ |\log |t||^n \int_{X_t} \Omega_t\wedge \overline{\Omega}_t } \Omega_t\wedge \overline{\Omega}_t
\\
=& \frac{ (L^n) |\log |t||^n }{  \int_{X_t} \sqrt{-1}^{n^2} \Omega_t\wedge \overline{\Omega}_t } |u_J|^2  \prod_i  \sqrt{-1} d\zeta_i\wedge d\bar{\zeta}_i,
\end{split}
\]
where $u_J$ is a holomorphic function of the defining functions $z_0, \ldots z_n$ of the divisors $E_i$, with limiting value $u_J(E_J)\neq 0$. The two expressions are matched by the condition that $\text{Log}_{\mathcal{X}* }d\mu_t$ converge to $d\mu_0$ as $t\to 0$, which boils down to
\[
(dd^c \phi_{CY,J,t})^n =n!\det(D^2\phi_0) \frac{ |u_J|^2}{|u_J(E_J)|^2 } \prod_i \frac{1}{4\pi} \sqrt{-1} d\zeta_i\wedge d\bar{\zeta}_i.
\]
Since $u_J$ has a Taylor expansion in $z_0,\ldots z_n$, we see that $\frac{ |u_J|^2}{|u_J(E_J)|^2 }=1+f$ for some smooth function $f$ in $\zeta_1,\ldots, \zeta_n$ with exponentially small $C^k$-norm bound
\[
\norm{ f}_{ C^k ( \text{Log}_{\mathcal{X} }^{-1}(W_\delta)  )  } \lesssim_k \exp(-c(W_\delta) |\log |t|| )
\]
for some exponent $c(W_\delta)>0$ depending on $W_\delta$.

We focus on balls in the local universal cover of $ \text{Log}_{\mathcal{X} }^{-1}(W_\delta)$ with definite size in the $\zeta_i$ coordinates. For sufficiently small $t$, then the volume relative error $f$ has arbitrarily small $C^k$-norm bound, and Theorem \ref{C0convergence} says the $C^0$-norm of $\phi_{CY,J,t} - \phi_0\circ \text{Log}_{\mathcal{X} }$ on the ball is also arbitrarily small. Thus we can apply Savin's theorem \ref{Savin}, to deduce that $\norm{ \phi_{CY,J,t} - \phi_0\circ \text{Log}_{\mathcal{X} } }_{C^k}$ is arbitrarily small on shrinked balls. Since $ W_\delta$ for varying $\delta$ give an exhaustion of the regular locus of $\phi_0$, this shrinking can be compensated by starting with a larger $W_\delta$, and we deduce the $C^k$-convergence estimate as required.
\end{proof}

The geometric meaning is that inside the generic region, the CY metric $\omega_{CY,t}$ is $C^\infty$-close to a \emph{semiflat metric}:
\begin{equation*}
\omega_{CY,t}\sim  dd^c \phi_0\circ \text{Log}_{\mathcal{X} }= \frac{1}{4\pi|\log |t||^2}  \sum_{1\leq i,j\leq n} \frac{\partial^2 \phi_0}{\partial x_i\partial x_j} \sqrt{-1} d\log z_i\wedge d\log \bar{z}_j .
\end{equation*}
In terms of the Riemannian metric tensors,
\begin{equation}\label{CYmetricasymptote}
g_{CY,t}\sim \frac{1}{2\pi|\log |t||^2} \text{Re}\{  \sum_{1\leq i,j\leq n} \frac{\partial^2 \phi_0}{\partial x_i\partial x_j}  d\log z_i\otimes d\log \bar{z}_j \}.
\end{equation}
The name `semiflat' means the metric restricted to the $T^n$-fibres are flat Euclidean. The $T^n$-fibres are precisely special Lagrangian in the model case
\[
\begin{cases}
\omega_{semiflat}= \frac{1}{4\pi|\log |t||^2}  \sum_{1\leq i,j\leq n} \frac{\partial^2 \phi_0}{\partial x_i\partial x_j} \sqrt{-1} d\log z_i\wedge d\log \bar{z}_j ,
\\
\Omega_{semiflat}= \text{const}\cdot \prod_1^n d\log z_i,
\end{cases}
\]
or equivalently \[
\text{Log}_{\mathcal{X}}: (z_1,\ldots z_n)\mapsto \frac{1}{\log |t| } (\log |z_1|,\ldots \log |z_n|   )
\]
is a special Lagrangian fibration in the model case for some choice of the phase angle. Since $(\omega_{CY,t},\Omega_t)$ is $C^\infty$-close to the model case,  standard perturbation theory allows one to perturb the $T^n$-fibres into special Lagrangians with respect to $(\omega_{CY,t},\Omega_t)$ in the generic region, to obtain a new special Lagrangian fibration. The details are carried out in \cite{Zhang}, and more expositions can be found in \cite{LiFermat}.

\begin{thm}
(Special Lagrangian fibration on the generic region)
For any given $0<\delta\ll 1$, then for $t$ sufficiently small depending on $\delta$, there is a special Lagrangian fibration on an open subset of $(X_t,\omega_{CY,t},\Omega_t)$ containing $W_\delta$.
\end{thm}

Consequently, assuming as always the comparison property between NA MA equation and real MA equation, then the special Lagrangian fibration exists on an open subset of arbitrarily large percentage of $X_t$ as $t\to 0$, which is the main theorem of the paper.

Finally we make a few comments about the status of the \emph{Kontsevich-Soibelman/Gross-Wilson conjecture}, which says that given a polarised algebraic  maximally degenerate family of CY manifolds, whose holonomy groups are exactly $SU(n)$, the Gromov-Hausdorff limit of the CY metrics $g_{CY,t}$ is the essential skeleton $Sk(X)$ equipped with a real Monge-Amp\`ere metric on the regular locus, the singular locus has real codimension 2, and $Sk(X)$ is homeomorphic to $S^n$.

What follows quickly from the metric asymptote (\ref{CYmetricasymptote}) and \cite{LiFermat} are the following facts, assuming the comparison property:

\begin{itemize}
\item  Over the regular locus of $\phi_0$ inside each $n$-dimensional open faces $\text{Int}(\Delta_J)$, the metrics $g_{CY,t}$ converge in the Gromov-Hausdorff sense to a real MA metric as $t\to 0$:
\[
g_{CY,t}\to \frac{1}{2\pi}  \sum_{1\leq i,j\leq n} \frac{\partial^2 \phi_0}{\partial x_i\partial x_j}  dx_i\otimes dx_j .
\] 
This is immediate from the much stronger $C^\infty_{loc}$ metric asymptote (\ref{CYmetricasymptote}).

\item  There is a uniform diameter bound $\text{diam}(X_t, g_{CY,t})\leq C$ \cite[Prop. 5.11]{LiFermat}\cite{LiTosatti}.

\item  Any point in $X_t\setminus W_\delta$ is within $C\delta^{1/2n}$-distance to a point on $W_\delta$ for sufficiently small $t$. This follows from the Bishop-Gromov comparison argument in \cite[section 5.3]{LiFermat}.

\item Consequently, the regular locus $W_0$ of $\phi_0$ inside the union of $n$-dimensional open faces of $Sk(X)$, is an open dense subset of any Gromov-Hausdorff limit space of $(X_t, g_{CY,t})$. 

\end{itemize}

\begin{rmk}
Notice there is a gap between the above results and the Gromov-Hausdorff convergence to the real MA metric on $Sk(X)$ defined by the Hessian of $\phi_0$, because the $n$-dimensional open faces are disconnected, and therefore we cannot access the distance of two points on different faces. One needs further information on the $(n-1)$-dimensional faces of $Sk(X)$.
\end{rmk}

What remains to be resolved are the following questions, which seem to contain substantial difficulty:

\begin{itemize}
\item 	Prove the comparison property.

\item  Formulate a global notion of convex functions and the real MA equation on $Sk(X)$, instead of just on the $n$-dimensional open faces. Notice this is nontrivial because $Sk(X)$ only has a piecewise affine structure, not a global affine structure. See section \ref{ConjecturalmeaningNAMAII} for some closely related discussions.

\item  Develop a regularity theory for such real MA metrics, and prove/disprove that the singular locus has real codimension at least two.  Notice this is false for real MA equations on the unit ball by a counterexample of Mooney \cite{Mooney}, so if it is true then there has to be a global reason.

\item The regularity theory should also show that $Sk(X)$ equipped with the real MA metric has the same topology as $Sk(X)$ viewed as a simplicial complex. This is nontrivial because a priori the real MA equation can have singularities which contract lines to points, and the singular set may even be quite fractal, such as in Mooney's example.

\item Prove an enhanced version of the comparison property between NA MA equation and real MA equation, which works globally on all faces of $Sk(X)$, not just on the $n$-dimensional open faces.

\item Extend the arguments in this paper over the global regular locus of $\phi_0$, to show that the CY metrics converge smoothly there as well. Use this to identify the Gromov-Hausdorff limit of $(X_t,g_{CY,t})$ with $Sk(X)$ equipped with the real MA metric defined by the Hessian of $\phi_0$.

\item  Show that $Sk(X)$ with the standard topology is homeomorphic to $S^n$. This question does not refer to the metric, and is  much studied in birational geometry \cite{NicaiseXu}\cite{NicaiseXuYu}. This can be checked explicitly for many examples. In general, it is known that $Sk(X)$ is a `pseudomanifold', its $\Q$-homology groups agree with $S^n$, and its fundamental group has trivial profinite completion, but the actual homeomorphism type is still elusive. 
\end{itemize}

\section{Further directions}\label{Furtherdirections}

Stepping outside of the main setting of this paper, we will mention some further problems, and make a few non-rigorous speculations. In particular, we will use K\"ahler geometric intuition to guess a formula for the NA measure over the lower dimensional faces of $\Delta_{\mathcal{X}}$ in terms of differential operators. We then suggest a possible link between the NA MA equation and degeneration of CY metrics, in \emph{non-maximal degeneration} settings, by proposing a generalized Calabi ansatz, which unifies the Calabi ansatz and the semiflat metric.

\subsection{Transcendental case}

 While this paper focuses on the algebraic case, one may wonder what happens for `transcendental families'. A prototypical examples is the family of degree $n+2$  hypersurfaces in $\mathbb{CP}^{n+1}$:
\[
X_t= \{    \sum_I a_I t^{\lambda_I} x^I=0      \} \subset \mathbb{CP}^{n+1},
\]
where $x^I$ denote the degree $n+2$ monomials, $a_I$ are coefficients chosen suitably generically, and $\lambda_I$ are exponents chosen suitably. When $\lambda_I$ are sufficiently irrational, this does not fit into our framework, yet the metric SYZ conjecture makes sense.

There seem to be two natural strategies. One is to make the NA pluripotential theory work over NA fields without a discrete valuation (\cf \cite{Boucksomnew2} for the latest progress), and the other is to develop the framework of real MA equation on polyhedral sets such as $Sk(X)$ without explicit reference to NA geometry.

\subsection{Conjectural meaning of the NA MA measure I}\label{ConjecturalmeaningofNAMA}

Let $(X,L)$ be an algebraic degeneration family with an ample polarization, and $(\mathcal{X},\mathcal{L})$ be a semistable snc model. This is not required to be CY, nor do we impose any maximal degeneration condition. It induces a formal model by base change. Consider the metric $\norm{\cdot}=\norm{\cdot}_{\mathcal{L}}e^{-\phi}$ on $(X_K,L)$, where we assume $\phi=\phi\circ r_{\mathcal{X}}$
and  $\phi$ is `sufficiently smooth' on $\Delta_\mathcal{X}$. Our plan is to use the heuristic logic at the end of section \ref{NAMAmeasure} to guess a formula for the NA measure over the interior of any face $\Delta_J\subset \Delta_{\mathcal{X}}$, in terms of \emph{differential operators}. The answer will involve an interesting correction factor to the real MA measure.

Let  $\Delta_J$ correspond to $E_J=\cap_{i\in J} E_i$ as usual. We first explain how to associate a class in $H^{1,1}(E_J)$  to each $x\in \text{Int}(\Delta_J)$. Let $x_i$ be the local affine coordinate corresponding to $E_i$ for any $i\in I$; we will only need $E_i\cap E_J\neq \emptyset$. We introduce an overparametrisation: let $u$ be a function of all $\{ x_i \}_{i\in I}$, such that $u$ agrees with $\phi$ at least on a neighbourhood of $\text{Int}(\Delta_J)\subset \Delta_{\mathcal{X}}$, and we assume $u$ is smooth. Here $x_i$ are treated as independent variables for $u$, even though on $\Delta_{\mathcal{X}}$ they satisfy various linear constraints.
 Consider the class in $H^{1,1}(E_J)$ defined by the affine linear combination of derivatives:
 \begin{equation}\label{H11class}
 \mathcal{D}_J(x, \norm{\cdot} )=c_1(\mathcal{L})-\sum_I \frac{\partial u}{\partial x_i} c_1(\mathcal{O}(E_i) ) 
 \end{equation}
 which depends only on $\phi$ because in $H^{1,1}(E_J)$
 \[
  c_1( \mathcal{O}(\sum_I E_i)) =c_1(\text{div}(dt))=0, \quad c_1(\mathcal{O}(E_i))=0 \quad \forall E_i\cap E_J=\emptyset.
 \]
Notice $\norm{\cdot}=\norm{\cdot}_{\mathcal{L} }e^{-\phi}$ and $\mathcal{D}_J(x,\norm{\cdot})$ are invariant under the change
\[
\mathcal{L}\to \mathcal{L}+\sum_I d_i E_i, \quad \phi\to \phi+ \sum_I d_i \phi_{E_i}
\]
where $d_i\in \R$ and $\phi_{E_i}$ denote the model functions associated to $E_i$. Thus the function $\mathcal{D}_J(x, \norm{\cdot} )$ on $\text{Int}(\Delta_J)$ is intrinsically associated to $\norm{\cdot}$.

Our strategy is to consider a family of Hermitian metrics $h_t$ on $L\to X_t$ such that $h_t^{ 1/|\log |t|| }\to \norm{\cdot}^2$ in the hybrid topology on $X\sqcup \Delta_{\mathcal{X}}$, and take the limit of the complex MA measures associated to the curvature forms of $h_t$. Introduce smooth Hermitian metrics $h_{\mathcal{L} }$ and $h_{E_i}$ on $\mathcal{L}$ and $\mathcal{O}(E_i)$ for $i\in I$. Use these to produce smooth functions $r_i$ on $\mathcal{X}$ for $i\in I$, such that near $E_i\subset \mathcal{X}$ the singularity is governed by $r_i\sim |z_i|e^{- \phi_i} $ for smooth local $\phi_i$, where $z_i$ are local defining equations for $E_i$, and away from $E_i$ the function $r_i$ is smooth and bounded positively from below.  In order for $h_t$ to have the appropriate convergence behaviour as $t\to 0$ around $\text{Int}(\Delta_J)\subset X\sqcup\Delta_{\mathcal{X}}$, we use the ansatz
\[
h_t\sim h_{ \mathcal{L} }\exp\left( 2\log |t| u(   \frac{\log r_i}{ \log |t|}  , i\in I    ) \right)
\]
where $r_i$ are regarded as smooth functions on $X_t$. The curvature form in $(X_t, c_1(L))$ is
\[
- dd^c\log h_t^{1/2} \sim -dd^c \log h_{\mathcal{L}}^{1/2}-  \log |t| dd^c u,
\]
\[
dd^c u= \sum_{i,j\in I} \frac{\partial^2 u}{\partial x_i\partial x_j} \frac{1}{|\log |t||^2}d\log r_i\wedge d^c \log r_j+ \frac{1}{\log |t|}\sum_I \frac{\partial u}{\partial x_i} dd^c \log r_i.
\]
Our goal is to extract the limiting contribution of $(-dd^c \log h_t^{1/2})^n$ to $\text{Int}(\Delta_J)$  as $t\to 0$. We need to separate this contribution from $\Delta_{J'}$ with $J'\supsetneq J$; algebraically this means taking the measure contribution near $E_J\subset \mathcal{X}$, but away from deeper strata $E_{J'}$. To formalize this, we consider the quantitative statum on $X_t$ (\cf section \ref{volumeasymptoteessentialskeleton})
\[
E_{J,t}^\epsilon= \{ q\in X_t| d(q, E_J)<\epsilon      \}\setminus \bigcup_{J'\supsetneq J} \{  q\in X_t| d(q, E_{J'}) <\epsilon \}, \quad \epsilon \ll 1.
\]
We need to compute
\[
\lim_{\epsilon\to 0} \lim_{t\to 0} \text{Log}_{\mathcal{X}* } \left( (-dd^c \log h_t^{1/2})^n \mres E_{J,t}^\epsilon \right)
\]
where the order of limit is important.

For fixed $\epsilon$, after deleting terms suppressed by order $O( \frac{1}{|\log |t||} )$, we have the asymptote on $E_{J,t}^\epsilon\subset X_t$ as $t\to 0$:
\[
dd^c u\sim \sum_0^p \frac{\partial^2 u}{\partial x_i\partial x_j} \frac{1}{|\log |t||^2} \frac{ \sqrt{-1}}{4\pi} d\log z_i\wedge d \log \bar{z}_j + \frac{1}{\log |t|}\sum_I \frac{\partial u}{\partial x_i} dd^c \log r_i,
\]
where $z_0,\ldots z_p$ are the local defining functions of the divisors $E_i\subset \mathcal{X}$, and $E_J=\cap_0^p E_i$. Using that $t=z_0\ldots z_p$ locally around $E_J$, we can eliminate the $ z_0$ variable to write
\[
dd^c u\sim \sum_1^p \frac{\partial^2 \phi}{\partial x_i\partial x_j} \frac{1}{|\log |t||^2} \frac{ \sqrt{-1}}{4\pi} d\log z_i\wedge d \log \bar{z}_j + \frac{1}{\log |t|}\sum_I \frac{\partial u}{\partial x_i} dd^c \log r_i,
\]
hence 
\begin{equation}\label{curvatureformasymptote}
	\begin{split}
- dd^c\log h_t^{1/2}\sim &\sum_1^p \frac{\partial^2 \phi}{\partial x_i\partial x_j} \frac{1}{|\log |t||} \frac{ \sqrt{-1}}{4\pi} d\log z_i\wedge d \log \bar{z}_j
\\
&-dd^c \log h_{\mathcal{L}}^{1/2} -\sum_I \frac{\partial u}{\partial x_i} dd^c \log r_i.	
	\end{split}
\end{equation}
Notice that $dd^c\log r_i\sim -dd^c \phi_i$ has a smooth extension to the central fibre as $t\to 0$; the singular effect is eliminated by $dd^c\log |z_i|=0$ on $X_t$. The term $\sum_1^p \frac{\partial^2 \phi}{\partial x_i\partial x_j} \frac{1}{|\log |t||} \frac{ \sqrt{-1}}{4\pi} d\log z_i\wedge d \log \bar{z}_j$ dominates in the directions transverse to $E_J$, and the term $-dd^c \log h_{\mathcal{L}}^{1/2} +\sum_I \frac{\partial u}{\partial x_i} dd^c \phi_i$ dominates in the directions tangential to $E_J$.

 The measure asymptote is now
\begin{equation}\label{measureasymptotegeneralisedCY}
\begin{split}
(-dd^c \log h_t^{1/2})^n \sim & p! \det( D^2\phi ) \left(  -dd^c \log h_{\mathcal{L}}^{1/2} +\sum_I \frac{\partial u}{\partial x_i} dd^c \phi_i    \right)^{n-p} \wedge \\
& \prod_{i=1}^p \frac{1}{2\pi|\log |t||}  d\log |z_i|\wedge d \arg{z}_i,
\end{split}
\end{equation} 
where $D^2\phi=( \frac{\partial^2 \phi}{\partial x_i\partial x_j}  )_{p\times p}$ is the Hessian matrix of $\phi$.
Thus the pushforward measure has the limit as $t\to 0$:
\[
\begin{split}
 \lim_{t\to 0} \text{Log}_{\mathcal{X}* } \left( (-dd^c \log h_t^{1/2})^n \mres E_{J,t}^\epsilon \right)= p!\det(D^2\phi)|dx_1\ldots dx_p| 
 \\
\times\int_{E_{J,0}^\epsilon}  \left(  -dd^c \log h_{\mathcal{L}}^{1/2} +\sum_I \frac{\partial u}{\partial x_i} dd^c \phi_i    \right)^{n-p} .
 \end{split}
\]
Taking $\epsilon\to 0$,
\[
\begin{split}
\lim_{\epsilon\to 0}\lim_{t\to 0} \text{Log}_{\mathcal{X}* } \left( (-dd^c \log h_t^{1/2})^n \mres E_{J,t}^\epsilon \right)= p!\det(D^2\phi)|dx_1\ldots dx_p| 
\\
\times\int_{E_J}  \left(  -dd^c \log h_{\mathcal{L}}^{1/2} +\sum_I \frac{\partial u}{\partial x_i} dd^c \phi_i    \right)^{n-p} .
\end{split}
\]
Here an interesting topological effect takes place. Even though $dd^c\log r_i$ starts life as an exact form on $X_t$, it acquires a first Chern class in the process of smooth extension to the central fibre $\mathcal{X}_0$, because we are removing the distributional contribution $dd^c\log |z_i|=[E_i]$. Thus on $E_J$, the smooth closed (1,1)-form
\[
-dd^c \log h_{\mathcal{L}}^{1/2} +\sum_I \frac{\partial u}{\partial x_i} dd^c \phi_i 
\]
lies in the $H^{1,1}$ class
\[
c_1(\mathcal{L}) - \sum_I \frac{\partial u}{\partial x_i} c_1(\mathcal{O}(E_i)) ,
\]
which is exactly the class $\mathcal{D}_J(x,\norm{\cdot})$ we introduced earlier (\cf (\ref{H11class})). We have thus obtained a formula for the double limit:
\[
 p!\det(D^2\phi)|dx_1\ldots dx_p| (\mathcal{D}_J(x,\norm{\cdot})^{n-p}\cdot E_J).
\]
Notice all auxiliary choices are eliminated at this stage.
According to our heuristic logic that the NA MA measure should be the limit of the corresponding complex MA measures on $X_t$, we conclude the \emph{heuristic formula for the NA MA measure} over $\text{Int}(\Delta_J)$
\begin{equation}\label{NAmeasureformula}
r_{\mathcal{X}*} MA( \norm{\cdot} )= p!\det(D^2\phi)|dx_1\ldots dx_p| (\mathcal{D}_J(x,\norm{\cdot})^{n-p}\cdot E_J).
\end{equation}
Notice the RHS is a differential operator in the potential $\phi$, because the intersection theoretic term 
$\mathcal{D}_J(x,\norm{\cdot})$ is affine linear in the first order derivatives of $\phi$. The formula exhibits a curious mixture of intersection theory with real MA operator.

\begin{rmk}\label{semipositivityconvexity}
(Semipositivity, convexity, nefness)
It is tempting to characterize the semipositivity condition on $\norm{\cdot}$, in terms of differential conditions on $\Delta_{\mathcal{X} }$, just like convex functions are characterised by the positivity of its Hessian matrix. We speculate that semipositivity should imply that for suitable choices of $h_\mathcal{L}$ and $r_I$, the curvature form $-dd^c \log h_t^{1/2}$ can be made positive up to small errors. In the $t\to 0$ limit, the formula (\ref{curvatureformasymptote}) then suggests that on each open face $\text{Int}(\Delta_J)$,
\begin{itemize}
\item The Hessian $D^2\phi\geq 0$, namely $\phi$ is convex;
\item The gradient satisfies that $\mathcal{D}_J(x,\norm{\cdot})\in H^{1,1}(E_J)$ lies in the nef cone. 
\end{itemize}
Do these two conditions completely characterize semipositive metrics with $\phi=\phi\circ r_{\mathcal{X}}$? If yes, it would naturally explain why (\ref{NAmeasureformula}) defines a measure, instead of just a signed measure.
\end{rmk}

\subsection{Conjectural meaning of NA MA II}\label{ConjecturalmeaningNAMAII}

In this section we will speculate on the concrete meaning of the Boucksom-Favre-Jonsson solution to the NA Calabi conjecture (\cf section \ref{NACalabi}). Let $(X,L)$ be a polarized algebraic degeneration family of Calabi-Yau manifolds, which needs not be a maximal degeneration. Let $(\mathcal{X},\mathcal{L})$ be a semistable snc model. An analogue of the Lebesgue measure $d\mu_0=\lim_{t\to 0}\text{Log}_{\mathcal{X}*}d\mu_t$ exists in this setting, which is supported on the essential skeleton $Sk(X)\subset \Delta_{\mathcal{X}}$. It can be found by the ideas outlined in section \ref{volumeasymptoteessentialskeleton}. To describe it, denote $m=\dim Sk(X)$, and $\Delta_J$ be any $m$-dimensional face of $Sk(X)$, corresponding to $E_J=\cap_0^m E_j$ with local defining functions $z_0,\ldots z_m$. The holomorphic volume form $\Omega$ on $\mathcal{X}$ induces a  holomorphic volume form on $E_J$, called its \emph{Poincar\'e residue} $\text{Res}_{E_J}(\Omega)$, determined by the equality in $K_{\mathcal{X}}|_{E_J}$:
\[
\Omega= \text{Res}_{E_J}(\Omega) \wedge td\log z_0\wedge \ldots d\log z_m.
\]
This is independent of the choice of local coordinates $z_i$. Then up to normalising $\Omega$ by a global multiplicative constant, on the interior of $\Delta_J$, 
\begin{equation}
d\mu_0= |dx_1\ldots dx_m| \cdot \int_{E_J} \sqrt{-1}^{(n-m)^2} \text{Res}_{E_J}(\Omega)\wedge \overline{\text{Res}_{E_J}(\Omega)} ,
\end{equation}
Notice it is proportional to $|dx_1\ldots dx_m|$, which explains the name `Lebesgue measure'.  The union of the $m$-dimensional open faces of $Sk(X)$ has the full $d\mu_0$-measure, which is one by our normalisation.

We are interested in the distinguished solution $\norm{\cdot}_{CY}=\norm{\cdot}_{\mathcal{L}}e^{-\phi_0}$ to the NA MA equation on $(X_K^{an},L)$ defined by
\begin{equation}
MA( \norm{\cdot}_{CY}  )=(L^n)d\mu_0.
\end{equation}
To give an interpretation, we boldly assume that $\phi_0=\phi_0\circ r_{\mathcal{X}}$, which can be regarded as a strong version of the comparison property. Then we are in the setting of section \ref{ConjecturalmeaningofNAMA}.

We begin with the interior of $m$-dimensional faces of $Sk(X)$, which can be viewed as the generic region. Comparing with the heuristic formula (\ref{NAmeasureformula}), and making regularity assumptions on $\phi_0$, we deduce a second order PDE
\begin{equation}\label{NAMAPDE}
\det(D^2\phi_0) (\mathcal{D}_J(x,\norm{\cdot}_{CY})^{n-m}\cdot E_J ) = \frac{(L^n)}{m!} \int_{E_J} \sqrt{-1}^{(n-m)^2} \text{Res}_{E_J}(\Omega)\wedge \overline{\text{Res}_{E_J}(\Omega)},
\end{equation}
where $(\mathcal{D}_J(x,\norm{\cdot}_{CY})^{n-m}\cdot E_J )$ defines a polynomial in the gradient of $\phi_0$.

On any other face $\text{Int}(\Delta_J)\subset \Delta_{\mathcal{X}}$ of dimension $p$, not necessarily on $Sk(X)$, 
\[
\det(D^2\phi_0) (\mathcal{D}_J(x,\norm{\cdot}_{CY})^{n-p}\cdot E_J )=0.
\]
There are two obvious mechanisms for this to happen:
\begin{itemize}
\item We may have the homogeneous real MA equation $\det(D^2\phi_0)=0$. There is a basic mechanism for this: notice that the model choice $\mathcal{X}$ is not canonical, and one can blow up to obtain higher models. For simple blowups, either one subdivides an existent face, or one creates new faces. On any new face $\det(D^2\phi_0)=0$ is automatic, because $\phi_0=\phi_0\circ r_{\mathcal{X}}$ implies that $\phi_0$ is independent of the new coordinate variable.

\item  We may have  $\mathcal{D}_J(x,\norm{\cdot}_{CY})^{n-p}\cdot E_J =0$, which is an algebraic condition on the gradient of $\phi_0$. The author speculates that this option happens for $(m-1)$-dimensional faces on $Sk(X)$, and its role is to match the gradients when we cross from one $m$-dimensional open face of $Sk(X)$ to another. In Remark \ref{semipositivityconvexity} we suggest $\mathcal{D}_J(x,\norm{\cdot}_{CY})$ may be a nef class on $E_J$, and our equation precisely says this class has zero volume. The birational geometric significance seems well worth investigating. We now discuss some of the simplest ways this mechanism could work:
\end{itemize}

\begin{eg}
When $n=m$, namely in the maximal degeneration case, for simplicity we consider an $(n-1)$-dimensional face $\Delta_J$ of $Sk(X)$, such that there are only two $n$-dimensional faces of $\Delta_{\mathcal{X}}$ containing $\Delta_J$, and they both lie on $Sk(X)$. Then the degree condition $\mathcal{D}_J(x,\norm{\cdot}_{CY})\cdot E_J =0$ imposes a matching condition on the gradient of $\phi_0$ across the $(n-1)$-dim face.

We consider a very concrete local example where  $\mathbb{P}^1\simeq E_J=\cap_1^{n}E_i$, 
\[
\deg \mathcal{O}(E_i)|_{E_J}=-d_i\leq 0, \quad i=1,2,\ldots n, \quad \sum_1^n d_i=2,
\]
and the two divisors $E_0, E_\infty$ intersect $E_J$ transversely at $0,\infty\in \mathbb{P}^1$, giving rise to the two $n$-dimensional faces. All divisors are reduced. The complex geometric local model for $\mathcal{X}$ comprises of two charts $\C^{n+1}_{z_0,\ldots z_n}$ and $\C^{n+1}_{w_0,\ldots, w_n}$. On the first chart $E_i=\{ z_i=0 \}$ for $i=0,1,\ldots n$, and $z_0$ is the affine coordinate on $\mathbb{P}^1\setminus \{\infty\}$.  On the second chart $E_i=\{ w_i=0\}$ for $i=1,\ldots n$, $E_\infty= \{ w_0=0 \}$, and $w_0$ is the affine coordinate on $\mathbb{P}^1\setminus \{0\}$. The transition on the overlap is
\[
w_0= z_0^{-1}, \quad w_i= z_i z_0^{d_i}, \quad i=1,2,\ldots n.
\]
The holomorphic volume form $\Omega\sim dz_0\wedge \ldots dz_n= -dw_0\wedge \ldots dw_n$, and the coordinate $t=z_0\ldots z_n=w_0\ldots w_n$. Locally on $X_t$
\[
\Omega_t\sim d\log z_1\wedge \ldots d\log z_n = -d\log w_1\wedge \ldots d\log w_n.
\]
On the two $n$-dimensional faces the coordinates are $x_i= \frac{\log |z_i|}{\log |t|}$ and $x_i'=\frac{\log |w_i|}{\log |t|}$ respectively, satisfying the linear constraints $\sum_0^n x_i=0$ and $\sum_0^n x_i'=0$, so we can eliminate $x_n, x_n'$. By adjusting $\mathcal{L}$ we can make it zero in the local model. The potential $\phi_0$ satisfies the real MA equation on the two $n$-dimensional faces:
\[
\det \left( \frac{\partial^2 \phi_0}{\partial x_i \partial x_j} \right)_{0\leq i,j\leq n-1}=\text{const}, \quad \det \left( \frac{\partial^2 \phi_0}{\partial x_i' \partial x_j'} \right)_{0\leq i,j\leq n-1}=\text{const},
\]
and the problem is to match them on $\Delta_J$. The complex geometry suggests that the domain of the coordinates $x_i, x_i'$ can be extended outside the original $n$-simplices, by the identificaton
\[
x_0'= -x_0, \quad x_i'= x_i + d_i x_0, \quad i=1,2,\ldots n-1,
\]
and the real MA equation is satisfied also on $\Delta_J$. In terms of gradients at $x\in \Delta_J$,
\[
\frac{\partial \phi_0}{\partial x_i'}= \frac{\partial \phi_0}{\partial x_i}, \quad i=1,\ldots n-1, \quad \frac{\partial \phi_0}{\partial x_0'}=- \frac{\partial \phi_0}{\partial x_0}+ \sum_1^{n-1} d_i \frac{\partial \phi_0}{\partial x_i}.
\]
The first $(n-1)$ conditions merely mean the tangent derivatives along $\Delta_J$ agree. The normal derivative matching condition precisely says
\[
\mathcal{D}_J (x,\norm{\cdot})\cdot E_J= \sum_1^{n-1} d_i \frac{\partial \phi_0}{\partial x_i}- \frac{\partial \phi_0}{\partial x_0}-  \frac{\partial \phi_0}{\partial x_0'}=0.
\]
From a different perspective, the zero degree condition explains why $\mathbb{P}^1$ can be invisible to the metric, so $\phi_0$ is allowed to remain smooth across $\text{Int}(\Delta_J)$.

\end{eg}

\begin{eg}
For $m= 1$, consider a $0$-dimensional face $\Delta_J$ of $Sk(X)$, such that there is a unique $1$-dimensional face  of $\Delta_{\mathcal{X}}$ containing $\Delta_J$, and it lies on $Sk(X)$. The simplest possibility $\mathcal{D}_J(x,\norm{\cdot}_{CY})=0$ imposes a  Neumann-like boundary condition on $\phi_0$.
This example may describe the Tian-Yau region in \cite{HSVZ}\cite{HSVZ2}.
\end{eg}

\begin{eg}
For $m=1$, consider a $0$-dimensional face $\Delta_J$ of $Sk(X)$, such that there are two $1$-dimensional faces  of $\Delta_{\mathcal{X}}$ containing $\Delta_J$, and they both lie on $Sk(X)$. A very simple geometric situation is when $E_J$ is the smooth total space of a possibly singular $\mathbb{P}^1$-fibration over an $(n-1)$-dim smooth variety $D$, and the two 1-dimensional faces correspond to two disjoint sections $E_{J',1}, E_{J',2}$ of the $\mathbb{P}^1$-fibration, so $E_{J',1}\simeq E_{J',2} \simeq D$. A natural way to make $\mathcal{D}_J(x,\norm{\cdot}_{CY})^{n}\cdot E_J =0$ and $\mathcal{D}_J(x,\norm{\cdot}_{CY})$ nef for $x\in \Delta_J$, is to ask $\mathcal{D}_J(x,\norm{\cdot}_{CY})$ to be the pullback of a nef class on the base $D$. We regard this nef class as the limiting element of $\mathcal{D}_J(y,\norm{\cdot}_{CY})\in H^{1,1}(E_{J',i})$
as $y$ approaches $x$ from either of the $1$-dimensional faces. Then the matching condition is naturally seen as the continuity of $\mathcal{D}_J(y,\norm{\cdot}_{CY})$ across the $0$-dimensional face. This example may be relevant for the Ooguri-Vafa type neck region in \cite{HSVZ}\cite{HSVZ2} (\cf  \cite[section 7.1]{HSVZ2}). The possibility for the $\mathbb{P}^1$-fibration to develop nodal fibres is related to the monopole bubbling phenemenon in these papers.

\end{eg}

\subsection{Non-maximal degenerations, generalised Calabi ansatz}

It is natural to ask what happens to the CY metrics for non-maximal polarized algebraic degenerations. Recall maximal degenerations require the essential skeleton to have dimension $n$. The other extreme case, where $\dim Sk(X)=0$, corresponds to degenerations with a uniform volume noncollapsing condition, and is by now quite well understood \cite{Zhang2}: the metric limit is a Calabi-Yau variety with klt singularities. The case $0<\dim Sk(X)<n$ is expected to exhibit a mixture of both the algebraic and NA/tropical behaviours. There is no systematic theory, but the CY metrics are described in some sporadic examples \cite{HSVZ}\cite{HSVZ2}. In such examples, a key role is played by a \emph{generalized Gibbons-Hawking ansatz}, which expresses CY metrics with some torus symmetry in terms of both symplectic moment coordinates and complex coordinates on the K\"ahler quotient. In the generic region, the asymptotic behaviour of these metrics is captured by the \emph{Calabi ansatz}, which describes the metric solely in complex coordinates using the K\"ahler potential \cite[chapter 2]{HSVZ2}. The procedure to pass from the Calabi ansatz to the generalized Gibbons-Hawking ansatz is akin to the Legendre transform.

The Calabi ansatz is as follows. Let $(Y,\Omega_Y)$ be a compact $(n-1)$-dimensional CY manifold with an ample line bundle $E$, and let $h$ denote the Hermitian metric on $E\to Y$ whose curvature form is the Calabi-Yau metric on $Y$ in the class $c_1(E)$, so the length $r=h^{1/2}$ defines a function on the total space $E$. 
Let $\xi$ denote a local  holomorphic fibre coordinate. The subset $\{  0<r\lesssim 1 \}\subset E$ has a nowhere vanishing holomorphic volume form $\Omega_E= \Omega_Y \wedge d\log \xi$, defined indpendent of $\xi$. Then the metric ansatz 
\[
\omega_E= dd^c (-\log r)^{(n+1)/n}
\]  
defines a CY metric on $\{  0<r\lesssim 1 \}\subset E$ compatible with $\Omega_E$.

Very little is known about the relationship between the NA MA solution and the metric degenerations, even in the aforementioned cases where the metric is well understood. Relating these is likely to require first proving some version of the comparison property, which is a major open problem. Notwithstanding all technical difficulties, we speculate the following heuristic principle: \emph{in the small $t$ limit, inside the generic region on $X_t$, the NA solution should approximately describe the behaviour of the K\"ahler potential at large scales, and obliterate the information on small compact directions.} Based on this principle, we will arrive at a generalised  Calabi ansatz, which is a candidate limiting description of the CY metrics in the generic region, at least in some idealized situations.

We assume the setting of section \ref{ConjecturalmeaningNAMAII}, and denote the NA MA solution as 
 $\norm{\cdot}_{CY}=\norm{\cdot}_{\mathcal{L}}e^{-\phi_0}$. Let $\Delta_J$ be an $m$-dimensional face of $Sk(X)$. We impose the simplifying assumptions:
 
 \begin{itemize}
 \item There is no $E_i$ among $i\in I$ intersecting $E_J$ transversely in $\mathcal{X}$. Consequently, the Poincar\'e residue $\text{Res}_{E_J}(\Omega)$ is nowhere vanishing on $E_J$. In this situation, the complex geometric picture around $E_J\subset \mathcal{X}$ is modelled on the total space of the rank $m$ vector bundle $\tilde{\pi}: \oplus_0^m \mathcal{O}(E_i) \to E_J$, with a trivialisation $ \mathcal{O}(\sum E_i)\simeq \mathcal{O}$ provided by the coordinate $t$. The holomorphic volume form for $|t|\ll 1$ is approximately
 \[
 \Omega\sim \tilde{\pi}^*\text{Res}_{E_J}(\Omega)\wedge t d\log z_0\wedge\ldots d\log z_m,
 \]
 so the holomorphic volume form $\Omega_t$ on $X_t$ is approximately
 \[
 \Omega_t \sim (-1)^{n-m} \tilde{\pi}^*\text{Res}_{E_J} \wedge d\log z_1\wedge\ldots d\log z_m.
 \]

 \item For every $x\in \text{Int}(\Delta_J)$, the class $\mathcal{D}_J(x, \norm{\cdot}_{CY})\in H^{1,1}(E_J)$ is K\"ahler (\cf Remark \ref{semipositivityconvexity}).

 \item The NA solution $\phi_0$ is smooth, and in particular strictly convex on $\text{Int}(\Delta_J)$.
 \end{itemize}

We recycle the computation in section \ref{ConjecturalmeaningofNAMA}, to construct the overparametrisation $u$ of $\phi_0$, and produce the Hermitian metric $h_t$ on $(X, c_1(L))$, so that $h_t^{1/|\log |t||}$ naturally converges to $\norm{\cdot}_{CY}^2$ in the hybrid topology. Up to $O(\frac{1}{|\log |t||})$ relative error, we have the metric asymptote (\ref{curvatureformasymptote}) and the measure asymptote (\ref{measureasymptotegeneralisedCY}).
This procedure is essentially dictated by the NA information. In particular, for each $x\in \text{Int}(\Delta_J)$, the class 
\[
[-dd^c \log h_{\mathcal{L}}^{1/2} +\sum_I \frac{\partial u}{\partial x_i} dd^c \phi_i ]=\mathcal{D}_J(x,\norm{\cdot}_{CY}) \in H^{1,1}(E_J)
\]
is fixed by the NA MA solution.

Notice at this stage we still have the freedom to adjust $h_{\mathcal{L}}$ and $r_i$ which enter the construction in section \ref{ConjecturalmeaningofNAMA}. This information controls the small scale metric, and is not directly visible to the NA MA solution. For each 
$x\in \text{Int}(\Delta_J)$, let $\phi_x$  be the solution to the complex MA equation on $E_J$:
\begin{equation}\label{generalisedCalabifibre}
\frac{ \left(  -dd^c \log h_{\mathcal{L}}^{1/2} +\sum_I \frac{\partial u}{\partial x_i} dd^c \phi_i + dd^c \phi_x    \right)^{n-m} }{ \mathcal{D}_J(x,\norm{\cdot}_{CY})^{n-m}  }=  \frac{  \text{Res}_{E_J}(\Omega)\wedge \overline{\text{Res}_{E_J}(\Omega)}  }{ \int_{E_J} \text{Res}_{E_J}(\Omega)\wedge \overline{\text{Res}_{E_J}(\Omega)}     }.
\end{equation}
If we wish to completely determine $\phi_x$ we need to fix a normalisation, but this choice is not essential. Under our hypotheses the dependence of $\phi_x$ on $x$ can be made smooth.

We then modify $h_t$ into $\tilde{h}_t= h_te^{-2\phi_x}$, where $x=\text{Log}_{\mathcal{X}} (z)$ on $X_t$. The metric asymptote (\ref{curvatureformasymptote}) then implies that on $\text{Log}_{\mathcal{X}}^{-1}(\text{Int}(\Delta_J))\subset X_t$ for $|t|\ll 1$, up to $O(\frac{1}{|\log |t||} )$ relative error,
\begin{equation}\label{generalizedCalabimetric}
\begin{split}
- dd^c\log \tilde{h}_t^{1/2}\sim &\sum_1^p \frac{\partial^2 \phi_0}{\partial x_i\partial x_j} \frac{1}{|\log |t||} \frac{ \sqrt{-1}}{4\pi} d\log z_i\wedge d \log \bar{z}_j
	\\
&-dd^c \log h_{\mathcal{L}}^{1/2} +\sum_I \frac{\partial u}{\partial x_i} dd^c \phi_i+ dd^c\phi_x.	
	\end{split}
\end{equation}
The normalisation ambiguity on $\phi_x$ is suppressed by $O(\frac{1}{|\log |t||} )$ relative to the term $\sum_1^p \frac{\partial^2 \phi_0}{\partial x_i\partial x_j} \frac{1}{|\log |t||} \frac{ \sqrt{-1}}{4\pi} d\log z_i\wedge d \log \bar{z}_j$. In particular, we see $- dd^c\log \tilde{h}_t^{1/2}$ is positive definite, so defines a K\"ahler metric.

We then compute its volume form up to $O(\frac{1}{|\log |t||})$ relative error:
\[
\begin{split}
(- dd^c\log \tilde{h}_t^{1/2})^n
\sim & m! \det( D^2\phi_0 ) \left(  -dd^c \log h_{\mathcal{L}}^{1/2} +\sum_I \frac{\partial u}{\partial x_i} dd^c \phi_i  +dd^c \phi_x  \right)^{n-m} \wedge \\
& \prod_{i=1}^m \frac{1}{2\pi|\log |t||}  d\log |z_i|\wedge d \arg{z}_i,
\\
\sim   m! \det( D^2\phi_0 )& \mathcal{D}_J(x,\norm{\cdot}_{CY})^{n-m} \frac{  \text{Res}_{E_J}(\Omega)\wedge \overline{\text{Res}_{E_J}(\Omega)}  }{ \int_{E_J} \text{Res}_{E_J}(\Omega)\wedge \overline{\text{Res}_{E_J}(\Omega)}     }\wedge
\\
& \prod_{i=1}^m \frac{\sqrt{-1}}{4\pi|\log |t||}  d\log z_i\wedge d \log \bar{z}_i
\\
\sim  \frac{ m!}{ (4\pi|\log |t||)^m }& \det( D^2\phi_0 ) \mathcal{D}_J(x,\norm{\cdot}_{CY})^{n-m}
 \frac{ \sqrt{-1}^{n^2} \Omega_t\wedge \overline{\Omega_t}  }{ \int_{E_J} \sqrt{-1}^{(n-m)^2} \text{Res}_{E_J}(\Omega)\wedge \overline{\text{Res}_{E_J}(\Omega)}     }.
\end{split}
\]
where we use the metric asymptote (\ref{generalizedCalabimetric}), the construction of $\phi_x$ (\ref{generalisedCalabifibre}), and the asymptote of $\Omega_t$ in terms of the Poincar\'e residue. Now applying the PDE interpretation of NA MA equation (\ref{NAMAPDE}), we see
\begin{equation}
(- dd^c\log \tilde{h}_t^{1/2})^n
\sim \frac{ (L^n)}{ (4\pi|\log |t||)^m }
 \sqrt{-1}^{n^2} \Omega_t\wedge \overline{\Omega_t}  .
\end{equation}
This means the metric $-dd^c\log \tilde{h}_t^{1/2}$ is approximately Calabi-Yau up to $O(\frac{1}{|\log |t||}  )$ relative error in the region on $X_t$ corresponding to the (slightly shrinked)
interior of $\Delta_J$. We call $-dd^c\log \tilde{h}_t^{1/2}$ the \emph{generalised Calabi ansatz}, and we expect this to model the CY metric on the generic region of $X_t$ in the class $c_1(L)$ up to small error. This provides a very tight relation between the NA MA equation and the degenerating CY metrics on $X_t$.

We now explain how to see the Calabi ansatz as a special case. The analogue of $\mathcal{X}$ is the total space of the rank 2 vector bundle $E\oplus E^{-1}\to Y$, so the det bundle has a canonical trivialisation coordinate $t$, which defines $X_t\subset E\oplus E^{-1}$. Equivalently $X_t$ can be viewed as a submanifold of $E$.  There is a holomorphic form on the total space $E\oplus E^{-1}$, given by $\Omega=\Omega_Y\wedge td\log z_0\wedge d\log z_1$, where $z_1$ is a local fibre coordinate on $E\to Y$, and $z_0z_1=t$. There is an induced holomorphic volume form $\Omega_t$ on $X_t$ defined by $\Omega=dt\wedge \Omega_t$, calculated to be $\Omega_t=(-1)^{n-1} \Omega_Y\wedge d\log z_1$, compatible up to sign with the Calabi ansatz setup. Now assume $E$ is ample, and is endowed with a Hermitian metric $h$ whose curvature form is the CY metric in $(Y, c_1(E))$. Denote $r=h^{1/2}$ as the length function on $E$. The analogue of $\mathcal{L}$ is $\mathcal{O}$. By the generalized Calabi ansatz prescription, we should find a function $\phi:\R_+\to \R$ satisfying the special case of the NA MA equation (\ref{NAMAPDE})
\[
\phi'' (\phi')^{n-1}=\text{const},
\]
and produce the metric $-\log |t| dd^c \phi( \frac{\log r}{\log |t|} )$ on $X_t$. The solution 
$
\phi(x)= x^{(n+1)/n} 
$
reproduces the Calabi ansatz up to scaling.

\begin{rmk}
The construction of the generalized Calabi ansatz above is analogous to the semi-Ricci-flat metric important in collapsing problems associated with holomorphic fibrations \cite{Tosatti}.
\end{rmk}

\begin{rmk}
In the maximal degeneration case $n=m$, near an $n$-dimensional open face of $Sk(X)$, there is no small compact directions described by holomorphic coordinates, and the generalized Calabi ansatz reduces to the semiflat metric. In general, this ansatz exhibits a mixture of holomorphic and NA behaviours.
\end{rmk}

\begin{rmk}
For $m<n$, in the non-generic regions corresponding to lower dimensional faces of $Sk(X)$, the NA MA solution is expected to lose information about the metric. One expects instead that Ooguri-Vafa type metrics constructed using the generalised Gibbons-Hawking ansatz become important \cite{HSVZ}\cite{HSVZ2}\cite{Li}. It would in particular be very interesting to understand the next generic behaviour, namely how the transition across $(m-1)$-dimensional faces of $Sk(X)$ occurs in general.
\end{rmk}

\end{document}